\documentclass[10pt]{article}  
\usepackage{amssymb}
\usepackage{amsmath,amsfonts,amsthm}

\newtheorem{theorem}{Theorem}[section]
\newtheorem{lemma}[theorem]{Lemma}
\newtheorem{proposition}[theorem]{Proposition}
\newtheorem{corollary}[theorem]{Corollary}

\newtheorem{example}[theorem]{Example}

\usepackage{fullpage}

\begin{document}

\numberwithin{equation}{section}

\title{ Ricci Almost Solitons on semi-Riemannian Warped Products}

\author{Valter Borges\footnote{Universidade de Bras\'ilia, Department of Mathematics, 70910-900, Bras\'ilia-DF, Brazil, valterb@mat.unb.br.Supported by CNPq Proc. 248877/2013-5  Ministry of Science and Technology, Brazil.  } \quad Keti Tenenblat\footnote{ Universidade de Bras\'{\i}lia,
 Department of Mathematics,
 70910-900, Bras\'{\i}lia-DF, Brazil, K.Tenenblat@mat.unb.br Partially supported by CNPq Proc. 312462/2014-0, Ministry of Science and Technology, Brazil.}
  }

\date{}

\maketitle{}

\begin{abstract}
It is shown that a gradient  Ricci almost soliton on a warped product, $(B^n\times_h F^m, g,f,\lambda)$
  whose potential function $f$ depends on the fiber,  is either a Ricci soliton  or $\lambda$ is not constant and the warped product, the base and the fiber are  Einstein manifolds, which admit conformal vector fields. 
  Assuming completeness, a classification is provided for the Ricci almost solitons on warped products, whose potential functions depend on the fiber.
  An important decomposition property of the potential function in terms of functions which depend either on the base or on the fiber is proven. In the case of a complete Ricci soliton, the potential function depends only on the base.

\end{abstract}
\vspace{0.2cm} 

 \section{Introduction}
 
 A {\it Ricci almost soliton} $(M,g,X,\lambda)$ is a Riemannian or semi-Riemannian manifold $(M,g)$ with a vector field $X$ and a smooth function $\lambda:M\rightarrow\mathbb{R}$ satisfying the following fundamental equation
 \begin{eqnarray*}
 	\displaystyle Ric+\frac{1}{2}\mathfrak{L}_{X}g=\lambda g.
 \end{eqnarray*}
 If the vector field $X$ is the gradient field of some function $f:M\rightarrow\mathbb{R}$, then the soliton is called a {\it gradient Ricci almost soliton}, or just {\it Ricci almost soliton}. In this case, it is denoted by  $(M,g,f,\lambda)$ and the fundamental equation becomes
 \begin{eqnarray}\label{eqriccisoliton}
 Ric+\nabla\nabla f=\lambda g,
 \end{eqnarray}
 where $f$ is called the \textit{potential function} and $\nabla\nabla f$ is the Hessian of $f$ with respect to the metric $g$. We say that a Ricci almost soliton is {\it shrinking, steady, expanding or undefined} if the function $\lambda$ is positive, null, negative or changes sign, respectively. The concept of Ricci almost soliton was introduced in \cite{pigola} generalizing the notion of Ricci soliton,  which is the case when the function $\lambda$ is constant. The importance of the Ricci solitons is due to their relation with the Ricci flow. In fact, they are stationary solutions of the Ricci flow, that was introduced by Hamilton \cite{hamilton1}. If the function $\lambda$ is not  constant, then Ricci almost solitons evolve under the Ricci flow changing only by conformal diffeomorphisms (see \cite{sharma} and \cite{tesenazareno} page 4). 
 Another relation with geometric flows is obtained by choosing  specific 
 functions for $\lambda$, for which  
 the corresponding Ricci almost solitons are self similar solutions of the so called Ricci-Bourguignon flow \cite{brozos}. This flow is an interpolation between the Ricci flow and the Yamabe flow \cite{catino2}. On the other hand Ricci almost solitons can be viewed as a generalization of Einstein manifolds \cite{besse}, as one can easily see by considering a constant function on an Einstein manifold.
 
Motivated by these relations with geometric flows, Ricci solitons and Einstein manifolds we are interested  in  investigating  the geometry of such manifolds and  we consider problems such as when a Ricci almost soliton becomes a Ricci soliton or even an Einstein manifold. In \cite{AltayGuler}, \cite{barros2}, \cite{barros3}, \cite{barros5}, \cite{grosh2}, \cite{HuangWei}, \cite{pigola} and \cite{shin} the authors proved that under certain geometric constraints a Ricci almost soliton becomes a Ricci soliton or an Einstein manifold carrying a conformal field.
 
Examples of Ricci almost solitons can be found in \cite{barros4}, \cite{feitosa2}, \cite{Nurowski} and \cite{pigola}. We observe that in \cite{pigola} there are examples of Riemannian manifolds that do not admit   Ricci almost soliton structures. For some results on Ricci almost solitons under certain geometric obstructions, we refer the reader to the  papers \cite{barros1}, \cite{barros2}, \cite{barros3}, \cite{barros4}, \cite{brozos}, \cite{calvino}, \cite{grosh1}  and references therein.
  
The warped products played an important role in the construction of non trivial Ricci solitons, as well as of Einstein manifolds. The cigar soliton, and more generally, the Bryant soliton, were obtained by using local warped product construction. More examples can be found in \cite{cao}, \cite{feitosa1}, \cite{sousa} and in references therein. It is important to note that warped products naturally arose in the studies of related structures on semi-Riemannian manifolds, as one can see for instance in \cite{AltayGuler}, \cite{HuangWei} and \cite{shin}.
 
Ricci almost solitons on warped products were studied firstly in \cite{feitosa2}. The authors gave a systematic approach using the hypothesis that the function $\lambda$ depends only on the base. Under this condition, they proved that the potential function $f$ depends only on the base. 

In this paper,  we start by characterizing  Ricci almost solitons on non trivial warped products in both situations, i.e., when the potential function depends on the fiber (Theorem \ref{withdependence}) and when it does not depend on the fiber (Theorem \ref{withoutdependence}), completing the previous study. We show that in both cases,  the fiber is necessarily an Einstein manifold. 

We then concentrate our main results in the case when the potential function depends on the fiber. In this case,  we show a rigity result in the sense that a Ricci almost soliton 
 $(B^n\times_{h}F^m,g,f,\lambda)$ reduces either to a Ricci soliton or $\lambda$ is not constant and $(B^n\times_{h}F^m,g)$ is an Einstein manifold (see Corollary \ref{rigidityresult}). Moreover, assuming completeness, Theorem \ref{coroclassificationnew} provides the classification of the Ricci almost solitons on warped products, whose potential functions depend on the fiber.
 
 These main results are obtained as a consequence of the following steps: 
 Proposition \ref{decompose} shows that for any Ricci almost soliton the potential 
 function $f$ decomposes in terms of functions that depend either on the base or on the 
 fiber. Then, using the characterization given in Theorem \ref{withdependence}, we prove  in Theorem \ref{impropercase}, that if the gradient of the warping function $h$ is an improper vector field, then the Ricci almost soliton is actually a Ricci soliton, the base is a Brinkman space and  the fiber is Ricci flat, moreover, we classify the fiber when it is complete.  When the gradient of the warping function is a proper vector field, then we prove in Theorem \ref{propercase} that it is either a Ricci soliton, i.e., $\lambda$ is constant or $\lambda$ is not constant and  $(B^n\times_{h}F^m,g)$, the base and the fiber are Einstein manifolds. Moreover, Theorems \ref{withdependence}  and  \ref{propercase} show that  the base, the fiber and the warped product admit conformal vector fields. The existence of such vector fields plays an important role in the classification of complete Riemannian and semi-Riemannian Einstein manifolds and it is used in the proof of Theorem 
\ref{coroclassificationnew}.

We observe that the assumption that the potential function depends only on the base is made in all constructions of Ricci solitons using warped product, as we can see in \cite{bryant} and \cite{danwan}, to name a few.  As a consequence of our approach, we conclude in Corollary \ref{nondependenceontehfiberriem} that the potential function of a Ricci soliton on a complete warped product depends only on the base, showing that this hypothesis can be eliminated. This result was also considered in \cite{sousa} with a different approach.  
   
This paper is organized as follows. In Section 2, we state our main results. Section 3, contains basic definitions and classical results needed for the proofs. In Section 4, we prove our main results. In the Appendix we collect all the results on conformal vector fields that will be used throughout the paper.



\section{Main Results}

\qquad In this section, we will state our main results. The proofs will be given in Section 4. In this paper, we are considering warped products $(B^n\times_h F^m,g)$ , where $(B^n,g_B)$ and $(F^m,g_F)$ are either Riemannian or semi-Riemannian manifolds and $g=g_B+h^2g_F$. We are  assuming  that $h$ is not constant. In this case we say that the warped product is not trivial. In what follows we will denote the connection, the Ricci curvature and other tensors defined using the metric $g_{B}$ with a subscript $B$, as $\nabla_B$, $Ric_{B}$. Similar notation will be considered for the metric $g_{F}$.

 Our first result says that for a Ricci almost soliton on a warped product, when the potential function depends on the fiber then the fundamental equation (\ref{eqriccisoliton}) on a warped product reduces to a system of equations on the base and on the fiber, in the following way:

\begin{theorem}\label{withdependence}
	Let $B^n\times_h F^m$ be a non trivial warped product where the base $(B^n,g_B)$ or 
the fiber $(F^m,g_{F})$ can be either a Riemannian or a semi-Riemannian manifold. Then 
$(B^n\times_{h}F^m,g,f,\lambda)$ is a Ricci almost soliton, with $f$ non constant on $F$
if, and only if, $f=\beta +h\varphi$, where $\varphi:F\rightarrow \mathbb{R}$ is not constant and  $\beta:B\rightarrow \mathbb{R}$ are differentiable functions such that 
	\begin{equation}\label{system_withdependence}
	\left\{ 
	\begin{array}[pos]{ll}
	\nabla_{B}\nabla_{B}h+ahg_{B}=0,\\\noalign{\smallskip}
	Ric_{B}+\nabla_{B}\nabla_{B}\beta=[h^{-1}(\nabla_{B}h)\beta-bh^{-1}+(n-1)a]g_{B}, \\\noalign{\smallskip}
	\nabla_{F}\nabla_{F}\varphi+(c\varphi+b)g_{F}=0,\\\noalign{\smallskip}
	Ric_{F}=(m-1)cg_{F},
		\end{array}
	\right.
	\end{equation}
	for some constants $a, \ b,\ c \in\mathbb{R}$, the function $\lambda$ is given by
\begin{equation}\label{lambda}
\lambda=h^{-1}(\nabla_{B}h)\beta-bh^{-1}+(m+n-1)a-ah\varphi,
\end{equation}
and the constants $a$ and $c$ are related to $h$ by the equation
	\begin{equation}\label{relation}
	|\nabla_{B}h|^{2}+ah^{2}=c.
	\end{equation}
\end{theorem}

Equations such as the first or  third equations  of (\ref{system_withdependence}) have appeared in many contexts. They appeared for example in concircular transformations \cite{kentaro}, in conformal transformations between Einstein spaces \cite{kunel} and in conformal vector fields on Einstein manifolds \cite{kunel2}.

A function satisfying equation (\ref{relation}) is said to have  constant energy, following \cite{castaneda}, where the author investigated properties of such functions. Equation (\ref{relation}) also appeared in the Critical Point Equation conjecture \cite{neto}.

As an application of Theorem \ref{withdependence} we can prove that for a complete warped product Ricci solitons (that is, when $\lambda$ is a constant) the potential function does not depend on the fiber.

\begin{corollary}\label{nondependenceontehfiberriem}
	Let $(B\times_{h}F,g,f,\lambda)$ be a Ricci soliton on a complete non trivial Riemannian or semi-Riemannian warped product. Then $f$ does not  depend on the fiber.
\end{corollary}

Corollary \ref{nondependenceontehfiberriem} was considered also in \cite{sousa} with a different approach. It shows that,  for complete Ricci solitons on semi-Riemannian warped product,  
the potential function depends only on the base.

Our next result  characterizes  Ricci almost solitons i.e., equation  (\ref{eqriccisoliton}),  on warped products,  when the potential function depends only on the base. 

\begin{theorem}\label{withoutdependence}
	Let $B^n\times_h F^m$ be a non trivial warped product where $(B^n,g_B)$ or 
	$(F^m,g_{F})$ can be either a Riemannian or a semi-Riemannian manifold. Then $(B^n\times_{h}F^m,g,f,\lambda)$ is a Ricci almost soliton, with $f$ constant on $F$ if, and only if,
	\begin{equation}\label{system_withoutdependence}
	\left\{ 
	\begin{array}[pos]{lll}
	Ric_{B}+\nabla_{B}\nabla_{B}f-mh^{-1}\nabla_{B}\nabla_{B}h=\lambda g_{B},\\\noalign{\smallskip}
	\lambda h^{2}=h(\nabla_{B}h)f-(m-1)|\nabla_{B}h|^{2}-h\Delta_{B}h+c(m-1),\\\noalign{\smallskip}
	Ric_{F}=c(m-1)g_{F},
	\end{array}
	\right.
	\end{equation}
	for some constant $c\in\mathbb{R}$.
\end{theorem}

The Riemannian version of Theorem \ref{withoutdependence} was also considered in \cite{feitosa2}, where the authors gave some explicit solutions to the system. The essence of both Theorems \ref{withdependence} and \ref{withoutdependence} is to express the condition for a warped product to be a Ricci almost soliton in terms of conditions on the base and on the fiber. Note that the first and third equations in Theorem \ref{withdependence} say that the corresponding gradient fields are conformal vector fields (see Appendix for definitions). In addition, the fourth equation of Theorem \ref{withdependence} and the third equation of Theorem \ref{withoutdependence} show that the fiber is an Eisntein manifold in both cases. 

We say that a semi-Riemannian manifold $(M,g)$ is a \textit{Brinkmann space} if it admits a parallel light like  vector field $X$, called a \textit{Brinkmann field}. These spaces play an important role in General Relativity \cite{brinkmann} and they were introduced by Brinkmann \cite{brinkmann} when the author studied conformal transformations between Einstein manifolds. 

We say that a vector field $X$ is \textit{improper} if there is an open set where $X$ is light like. If there is no such an open set then $X$ is called a \textit{proper} vector field. Our next two results show the rigidity of a Ricci almost soliton on a warped product 
$(B^n\times_{h}F^m,g,f,\lambda)$, 
when the potential function depends on the fiber. Namely, we show that such a Ricci almost soliton is either a Ricci soliton (i.e. $\lambda$ is constant) or $\lambda$ is not constant but $(B\times_h F,g)$ is an Einstein manifold.

  Theorems \ref{impropercase} and \ref{propercase} consider respectively  the case when $\nabla_Bh$ is an improper vector field and a proper one and the potential function of the warped product depends  on the fiber.

\begin{theorem}\label{impropercase}
	Let $B^n\times_h F^m$, $n\geq 2$, be a non trivial warped product where the base $(B^n,g_B)$ is a semi-Riemannian manifold and the fiber 	$(F^m,g_{F})$ can be either a Riemannian or a semi-Riemannian manifold. Then $(B^n\times_{h}F^m,g,f,\lambda)$ is a Ricci almost soliton, with $f$ non constant on $F$ and $\nabla_{B}h$ an improper vector field on $B$ if, and only if, $\lambda$ is constant, i.e. it is a Ricci soliton and $f=\beta +h\varphi$, where $\varphi:F\rightarrow \mathbb{R}$ non constant and $\beta:B\rightarrow\mathbb{R}$ are smooth functions satisfying
	\begin{equation*}
	g(\nabla_{B}h,\nabla_{B}\beta)=\lambda h+b, \qquad Ric_{B}+\nabla_{B}\nabla_{B}\beta=\lambda g_{B}, \qquad \nabla_{F}\nabla_{F}\varphi+bg_{F}=0
	\end{equation*}
for a constant $b\in \mathbb{R}$, $B$ is a Brinkmann space with $\nabla_{B}h$ as a Brinkmann field
	and $F$ is Ricci flat. 
	
	If in addition $F$ is complete, then it is isometric to 
	\begin{enumerate}
		\item $\pm\mathbb{R}\times\bar{F}^{m-1}$, where $\bar{F}$ is Ricci flat, if $b=0$;
		\item $\mathbb{R}^{m}_{\epsilon}$, if $b\neq0$.
	\end{enumerate}
\end{theorem}

The vector field $\nabla_{B}h$ is \textit{homothetic} if its local flow acts by translations. Otherwise it is called \textit{non-homothetic}. 

\begin{theorem}\label{propercase}
	Let $B^n\times_h F^m$ be a non trivial warped product where the base $(B^n,g_B)$ or the fiber $(F^m,g_{F})$ can be either a Riemannian or a semi-Riemannian manifold and suppose that $(B^n\times_{h}F^m,g,f,\lambda)$ is a Ricci almost soliton with $f$ non constant on $F$ and $\nabla_{B}h$ a proper vector field. Then

i) If 	$\nabla_B h$ is homothetic, then $\lambda$ is constant, i.e, it is a Ricci soliton.

ii) If  $\nabla_B h$ is non-homothetic,   then $\lambda$ is not constant,    
	 $B^n\times_h F^m$, $B$ and $F$ are Einstein manifolds such that 
 \begin{equation}\label{RicWarpBF}
Ric_{B\times_hF}=a(n+m-1)g, \qquad 
Ric_B=a(n-1)g_B,\qquad \qquad Ric_F=c(m-1)g_F, 
\end{equation}
where the constants $a\neq 0$ and $c$ are related to $h$ by $|\nabla_B h|^2+ah^2=c$. Moreover,  
 $\nabla f$ and $\nabla_{B}h$ are conformal gradient fields on $B^n\times_h F^m$ and on $B^{n}$, respectively, satisfying
 \begin{equation}\label{conformalfh}
	\nabla\nabla f+(af+a_{0})g=0,\qquad \qquad 
	\nabla_{B}\nabla_{B} h+ahg_{B}=0.
\end{equation}
and  
\begin{equation}\label{lambdaT26}
	\lambda=-af+a(m+n-1)-a_{0}, 
\end{equation}
for some constant $a_{0}\in\mathbb{R}$. 
\end{theorem}

A direct corollary of both Theorem \ref{impropercase} and Theorem \ref{propercase} is the following rigidity result.   
Other  rigidity results  can be found  in  \cite{barros2}, \cite{barros3}, \cite{grosh2} or \cite{pigola}.

\begin{corollary}\label{rigidityresult}
	If $(B^n\times_{h}F^m,g,f,\lambda)$ is a warped product Ricci almost soliton,  with $f$ non constant on $F$, then one of the following holds
	
	i) $\lambda$ is constant, i.e., it is a Ricci soliton;  
    
    ii)	$\lambda$ is not constant, $(B^n\times_{h}F^m,g)$, $(B,g_B)$ and $(F,g_F)$  are  Einstein manifolds,  $\nabla_{B}h$ is a proper and non-homothetic conformal vector field,     and $\nabla f$, $\nabla_{F} \varphi$ are conformal.
\end{corollary}

Observe that Corollaries \ref{rigidityresult} and \ref{nondependenceontehfiberriem} imply   
that a complete Ricci almost soliton $(B^n\times_{h}F^m,g,f,\lambda)$ on a warped product,   with $f$ non constant on the fiber $F$, will have  necessarily non constant $\lambda$.   
We conclude this section  with a classification result for such Ricci almost solitons.  In order to do so, we consider the following classes of $n$-dimensional complete Einstein manifolds (see Theorems \ref{conformal2}-\ref{conformal3new} in the Appendix):\\

\vspace{.05in}

\noindent{\bf Class I}
\begin{enumerate}
\item  $\mathbb{R}\times N^{n-1}$ where $(N,g_N)$ is a complete Riemannian or semi-Riemannian Einstein manifold.
\item A Brinkman space of dimension $n\geq 3$, i.e. a semi-Riemannian manifold $(M^n,g)$ admitting a parallel light like vector field.
\end{enumerate}  
{\bf Class II} 
\begin{enumerate}
\item $\mathbb{S}^{n}_{\varepsilon}(1/\sqrt{c})$,  when $0\leq\varepsilon\leq n-2$; the covering of $\mathbb{S}^{n}_{n-1}(1/\sqrt{c})$ when $\varepsilon=n-1$ and  the upper part of $\mathbb{S}^{n}_{n}(1/\sqrt{c})$ when $\varepsilon=n$ with $c>0$.
\item  $\mathbb{H}^{n}_{\varepsilon}(1/\sqrt{|c|})$, when  $2\leq\varepsilon\leq n-1$;  the  covering of $\mathbb{H}^{n}_{1}(1/\sqrt{|c|})$ when $\varepsilon=1$ and the upper part of $\mathbb{H}^{n}_{0}(1/\sqrt{|c|})$ when $\varepsilon=0\,$ , with $c<0$.  				
\item $(\mathbb{R}\times N^{n-1},\pm dt^{2}+\cosh^{2}(\sqrt{|c|}\,t)g_{N})$, where  $(N^{n-1},g_{N})$ is a RIemannian or semi-Riemannian Einstein manifold.
\item $(\mathbb{R}\times N^{n-1},\pm dt^{2}\pm e^{2\sqrt{|c|}\,t}g_{N})$, where  $(N^{n-1},g_{N})$ is a Riemannian Einstein manifold, 
\end{enumerate}
	
The following result classifies the  complete Ricci almost solitons on warped products, whose potential functions depend on the fiber. It also shows that $\lambda$ depends on the fibre. 	
	
\begin{theorem}\label{coroclassificationnew}
	Let $M^{n+m}=B^n\times_h F^m$ be a non trivial warped product where $(B^n,g_B)$ or $(F^m,g_{F})$ can be either a Riemannian or a semi-Riemannian manifold. Then $(B^n\times_{h}F^m,g,f,\lambda)$ is a complete Ricci almost soliton with $f$ non constant on $F$ if, and only if, there exist  constants $a\neq 0,\	
	a_0, c \in\mathbb{R}$ such that  $f=a^{-1}(-\lambda+a(m+n-1)-a_{0})$ and 
\begin{enumerate}
\item if $n=1$ then $B^1$ is isometric to  $(\mathbb{R},sgn\, a \,dt^2)$  
\begin{equation}\label{hc}
h=\left\{ 
\begin{array}{lll}
Ae^{\sqrt{|a|}t} & \mbox{ if } & c=0,\\ \noalign{\smallskip}
\sqrt{|\frac{c}{a}|}[ \cosh(\sqrt{|a|}t+ B)] & \mbox{ if } & c\neq 0,
\end{array} 
 \right.
 \end{equation}
where $A\neq 0$ and $B\in \mathbb{R}$. Moreover, 
$M$ is an Einstein manifold satisfying $Ric_{M}=(m+n-1) ag$ and if $m\geq 2$, F is an Einstein manifold satisfying $Ric_F=(m-1)cg_F$.
\item If $n\geq 2$ and $m\geq 2$ then 
\begin{itemize}
\item $M^{n+m}$ is an Eisntein manifold isometric either to a manifold of Class II.1 (resp. II.2) when $a>0$ (resp. $a<0$) and $f$ has some critical point  or it is isometric to a manifold of Class II.3 or II.4 if $f$ has no critical points. 
\item  $B$ is a complete Einstein manifold isometric either to a manifold of 
Class II.1 (resp. Class II.2) and index $\varepsilon_B=n$ (resp. $\varepsilon_B=1$)  if $a>0$ (resp. $a<0$) and $h$ has critical points   
or to a manifold of Class II.3 or II.4 if $h$ has no critical points.
\item $F$ is a complete Einstein manifold isometric to either $\mathbb{R}^n_\varepsilon$, or to a manifolds of Class I when  $c=0$ and it is isometric to a manifold of Class II when $c\neq 0$.  
\end{itemize}
\item Moreover, $F^{m}$, $m\geq 1$ is positive definite (resp. negative definite) if $B^n$, $n\geq 1$ is positive definite (resp. negative definite). 
\end{enumerate}
\end{theorem}

\noindent {\bf Remarks:} 
\begin{enumerate}
\item As we will see in the next sections, the proofs of our main results 
rely strongly on an important decomposition property of the potential function, namely 
$f=\beta+h\varphi$, where $\beta$ and $h$ are defined on the base and $\varphi$ is defined on the fiber $F$ (see Proposition \ref{decompose}). By considering this decomposition, in Theorem  \ref{coroclassificationnew} item 2, when $c\neq 0$, the fiber $F$ is isometric to a manifold of Class II 1 (resp. Class II 2) when $c>0$ (resp. $c<0$) and $\varphi$ has some critical point, 
while $F$ is isometric to a manifold of Class II 3 or 4 when $\varphi$ has no critical points
(see proof of Theorem \ref{coroclassificationnew}). 
\item 
We observe that, when we are in the Riemannian setting, Theorem \ref{impropercase} does not occur. Moreover, Class I only contains the product of $\mathbb{R}\times N^{n-1}$, where $(N,g_N)$ is a complete Riemannian Einstein manifold and Class II is restricted to the Riemannian manifolds. The classification in the Riemannian case was first obtained in 
\cite{pigola}.   
\end{enumerate}


\section{Preliminaries}

In this section we recall some definitions and results that will be used in 
Section 4, for  the proofs of the main results.

\subsection{Warped Products}

Consider two semi-Riemannian manifolds $(B^{n},g_{B})$ and $(F^{m},g_{F})$. Given a smooth function $h:B\rightarrow(0,+\infty)$, we can consider the warped product $B\times_{h}F$, see \cite{bishop} or \cite{oneill}, with warping function $h$, as the product manifold $B\times F$ endowed with the metric $g=g_{B}+h^{2}g_{F}$, defined by
\begin{equation}
g=\pi^{*}g_{B}+(h\circ\pi)^{2}\sigma^{*}g_{F},
\end{equation}
where $\pi:B\times F\rightarrow B$ and $\sigma:B\times F\rightarrow F$ are the canonical projections. So $B\times_{h}F$ is a semi-Riemannian manifold of dimension $n+m$.

In what follows, we will consider on the product lifted vector fields from the base and from the fiber identifying these vector fields with the corresponding vector fields on the base and on the fiber, respectively. The set of all such liftings from the base will be denoted by $\mathfrak{L}(B)\subset\mathfrak{X}(B\times F)$ and the set of all liftings from the fiber will be denoted by $\mathfrak{L}(F)\subset\mathfrak{X}(B\times F)$. Vector fields lifted from the base will be denoted by $X, Y, Z\in\mathfrak{L}(B)$ and vector fields lifted from the fiber will be denoted by $U, V, W\in\mathfrak{L}(F)$. For more information about lifting vector fields see for example \cite{oneill}.

The propositions below can be found in \cite{bishop} or \cite{oneill}. They express the geometry of the warped product in terms of the base and fiber geometries and the properties of the warping function. They can be used to produce examples of metrics satisfying some prescribed properties, as one can see for example in \cite{bishop}.  


\begin{proposition}\cite{bishop}\label{prop_ricciwarped}
	Let $M= B^n\times_{h}F^m$ be a Riemannian or semi-Riemannian warped product. Then the Ricci tensor of $M$ is given by
	\begin{equation}\label{ricciwarped}
	\left\{
	\begin{array}[pos]{lll}
	Ric(X,Y)=Ric_{B}(X,Y)-mh^{-1}\nabla_{B}\nabla_{B}h(X,Y),\\\noalign{\smallskip}
	Ric(X,U)=0,\\\noalign{\smallskip}
	Ric(U,V)=Ric_{F}(U,V)-[h\Delta_{B}h+(m-1)|\nabla_{B}h|^{2}]g_{F}(U,V).
	\end{array}
	\right.
	\end{equation}
\end{proposition}

The next result is a direct consequence of Proposition \ref{prop_ricciwarped} and it will be useful for the proofs of our main results.

\begin{proposition}\cite{kimkim}\label{einsteincondition}
A semi-Riemannian warped product, $B^{n}\times_{h}F^m$, is an Einstein space with Einstein constant $a\in\mathbb{R}$ if, and only if, there is a constant $c\in\mathbb{R}$ so that
	\begin{equation}
	\left\{
	\begin{array}[pos]{lll}
	Ric_{B}-mh^{-1}\nabla_{B}\nabla_{B}h=a(m+n-1)g_{B},\\\noalign{\smallskip}
	h\Delta_{B}h+(m-1)|\nabla_{B}h|^{2}+a(m+n-1)h^{2}=c(m-1),\\\noalign{\smallskip}
	Ric_{F}=c(m-1)g_{F}.
	\end{array}
	\right.
	\end{equation}
\end{proposition}

If the base $B$ is a connected interval 
$I\subset\mathbb{R}$, then  Proposition 
\ref{einsteincondition} takes a simpler form,  
that we state below,  for future references.

\begin{corollary}\label{einsteinonedimentional}
A semi-Riemannian warped product of the form 
$I\times_{h}F^m$, where $I\subset \mathbb{R}$, is an Einstein 
space if, and only if, $(F^{m},g_{F})$ is an Einstein space and the function $h$ satisfies
	\begin{equation}
	h''\pm ah=0\qquad\text{and}\qquad\pm (h')^{2}+ah^{2}=c, 
	\end{equation}
	where $a$ is the Einstein constant of $\, I\times_{h}F^m$ and $c$ is the Einstein constant of $F$.
\end{corollary}

As an immediate consequence of the properties of the connection in a warped product, proved in \cite{bishop}, one obtain 

\begin{proposition}\label{prop_hesswarped}
	Let $M=B^n\times_{h}F^m$ be a semi-Riemannian warped product. Then the Hessian of a function $f:M\rightarrow\mathbb{R}$ is given by
	\begin{equation}\label{hesswarped}
		\left\{
		\begin{array}[pos]{lll}
			\nabla\nabla f(X,Y)=\nabla_{B}\nabla_{B}f(X,Y),\\\noalign{\smallskip}
			\nabla\nabla f(X,U)=X(U(f))-h^{-1}X(h)U(f),\\\noalign{\smallskip}
			\nabla\nabla f(U,V)=\nabla_{F}\nabla_{F}f(U,V)+h(\nabla_{B}h)f g_{F}(U,V).
		\end{array}
		\right.
	\end{equation}
\end{proposition}	

By a complete semi-Riemannian manifold we mean a semi-Riemannian manifold where each geodesic can be extended  to $\mathbb{R}$. 
In the Riemannian case one shows the following.

\begin{proposition}\cite{bishop}\label{completeriemannian}
	A Riemannian warped product $B\times_{h}F$ is complete if, and only if, $B$ and $F$ are complete.
\end{proposition}

Been and Busemann showed that $(\mathbb{R}\times\mathbb{R}, dx^2-e^{2x}dy^2)$ is not a complete semi-Riemannian manifold. In fact, they showed that there are light like geodesics that can not be extended to  $\mathbb{R}$, see \cite{oneill} (page 209). 
Their example shows that there is no result similar to Proposition \ref{completeriemannian} for  indefinite signature.

For our purposes we have the following result that guarantees the non completeness of the semi-Riemannian warped product, whenever the gradient of the warping function is a parallel vector field on the base. For more results on completeness of semi-Riemannian manifolds see \cite{semiriemanniancomplete}

\begin{proposition}\cite{kimkim}\label{nondependenceontehfiber}
	Let $B\times_{h}F$ be a non trivial warped product, where $(B^n,g_B)$ or $(F^m,g_{F})$ can be either a Riemannian or a semi-Riemannian manifold. If $\nabla_{B}h$ is a parallel vector field on $B$, then $B\times_{h}F$ is not complete.
\end{proposition}
\begin{proof}
	Suppose by contradiction that $B\times_{h}F$ is complete. Consider $p_{0}\in B$ and $v_{0}\in T_{p_{0}}B$ such that $dh_{p_{0}}v_{0}\neq0$. Let  $\gamma$ be the geodesic such that $\gamma(0)=p_{0}$ and $\gamma'(0)=v_{0}$. Since $\nabla_{B}h$ is parallel, it follows that 	
	\begin{equation*}
	\begin{array}[pos]{lll}
	(h\circ\gamma)''(t)&=\gamma'(\gamma'(h))
		=\gamma'(\gamma'(h))-
\nabla_{B_{\gamma'}}\gamma'(h)\\\noalign{\smallskip}
	&=\nabla_{B}\nabla_{B}h(\gamma',\gamma')
	=0.
	\end{array}
	\end{equation*}
	Therefore, there  exist constants $a_{0}, b_{0}\in\mathbb{R}$, so that 
	\begin{equation*}\label{affineexpression}
	(h\circ\gamma)(t)=a_{0}t+b_{0}.
	\end{equation*}
	Observe that
	\begin{equation*}
	a_{0}=(h\circ\gamma)'(0)
	=dh_{p_{0}}v_{0}
	\neq 0.
	\end{equation*}
By assumption $\gamma$ is defined on $\mathbb{R}$, hence we may consider $t_{0}=-b_{0}/a_{0}\in\mathbb{R}$. However, $ h(\gamma(-b_0/a_0))=0$,  which contradicts the fact that $h\neq 0$.
\end{proof}


\subsection{Bochner Formula}

In this section we will state a version of the Bochner formula that will be used in the next section. For a proof in the Riemannian case, see Lemma 2.1 of \cite{petersen}. We observe that the same proof is valid for any signature.

\begin{theorem}[\cite{petersen}]
	Let $(M,g)$ be a Riemannian or semi-Riemannian manifold and let $\varphi:M\rightarrow\mathbb{R}$ be a smooth function. Then
	\begin{equation}\label{bochner}
	div(\nabla\nabla\varphi)(X)=Ric(\nabla\varphi,X)+X(\Delta\varphi),
	\end{equation}
	for all $X\in\mathfrak{L}(M)$.
\end{theorem}
With this version of Bochner formula, we can provide a simple proof of the proposition below when $n\geq2$. For another proof when $n\geq3$ see (\cite{kunel2}).
\begin{proposition}\label{otherprove}
	Let $(M^{n},g)$ be an Einstein manifold with dimension $n\geq2$ and Einstein constant $a$. If $\varphi:M\rightarrow\mathbb{R}$ is a smooth function such that $\nabla\varphi$ is a conformal vector field satisfying
	\begin{equation*}
	\nabla\nabla\varphi+\phi g=0
	\end{equation*}
	for  some smooth function $\phi:M\rightarrow\mathbb{R}$, then there is a constant $b\in\mathbb{R}$ such that $\phi=-a\varphi-b$.
\end{proposition}
\begin{proof}
	It is  easy to see that $\Delta\varphi=n\phi$ and that $div(\nabla\nabla\varphi)(X)=X(\phi)$, for all $X\in\mathfrak{X}(M)$. Using Bochner formula, we have
	\begin{equation*}
	(n-1)X(\phi+a\varphi)=0.
	\end{equation*}
	Since $X$ is an arbitrary field and $n\geq2$, it follows that there is a constant $b$ satisfying the assertion.
\end{proof}


\section{Proof of the main results}\label{Proofofthemainresults}
We start with an important decomposition property of the potential function of a  Ricci almost soliton on a warped product. We prove that the potential function decomposes in terms of functions which depend either on the base or on the fiber.
 
\begin{proposition}\label{decompose}
Let $(B^n\times_{h}F^m,g,f,\lambda)$ be a Ricci almost soliton defined on a warped product manifold, where the base $(B^n,g_{B})$ or the fiber  $(F^m,g_{F})$ are either Riemannian or semi-Riemannian manifolds, $h:B\rightarrow\mathbb{R}$ is a positive smooth function and $g=g_B+h^2g_F$. Then the  potential function $f$ can be decomposed as 
	\begin{equation}\label{property}
	f=\beta+h\varphi,
	\end{equation}
	where $\beta:B\rightarrow\mathbb{R}$ and $\varphi:F\rightarrow\mathbb{R}$ are smooth functions and  the fundamental equation (\ref{eqriccisoliton}) is equivalent to the system
	\begin{equation}\label{system_characterization}
	\left\{ 
	\begin{array}[pos]{ll}
	Ric_{B}+\nabla_{B}\nabla_{B}\beta + (\varphi-mh^{-1})\nabla_{B}\nabla_{B}h=\lambda g_{B},\\
	Ric_{F}+h\nabla_{F}\nabla_{F}\varphi=[h\Delta_{B}h+(m-1)|\nabla_{B}h|^{2}-h(\nabla_{B}h)\beta-\varphi h(\nabla_{B}h)h+\lambda h^{2}]g_{F}.
	\end{array}
	\right.
	\end{equation}
\end{proposition}

\noindent \begin{proof}
In view of Proposition \ref{prop_ricciwarped} and  Proposition \ref{prop_hesswarped} we can rewrite the fundamental equation (\ref{eqriccisoliton}) as follows
\begin{equation}\label{eqsolitonwarped}
\left\{ 
\begin{array}[pos]{lll}
Ric_B(X,Y)-m h^{-1}\nabla_{B}\nabla_{B}h(X,Y)+\nabla_{B}\nabla_{B}f(X,Y)=\lambda g_{B}(X,Y),\\\noalign{\smallskip}
Ric_{F}(U,V)+\nabla_{F}\nabla_{F}f(U,V)=[\lambda h^2+(m-1)h^{-2}|\nabla_{B}h|^{2}+h^{-1}\Delta_{B}h-h(\nabla_{B}h)f] g_{F}(U,V),\\\noalign{\smallskip}
X(U(f)) = h^{-1}X(h)U(f).
\end{array}
\right.
\end{equation}
Observe that $X(U(f))-h^{-1}U(f)X(h)=0$ implies
\begin{equation*}\label{equation_to_modify}
\begin{array}[pos]{lll}
X(U(fh^{-1}))&=& X(U(f)h^{-1})\\\noalign{\smallskip}
&=& X(U(f))h^{-1}-U(f)h^{-2}X(h)\\\noalign{\smallskip}
&=&0,
\end{array}
\end{equation*}
for all $X\in\mathfrak{L}(B)$ and all $U\in\mathfrak{L}(F)$. 

Therefore, there are smooth functions $\beta:B\rightarrow\mathbb{R}$ and $\varphi:F\rightarrow\mathbb{R}$ such that the potential function $f$ decomposes as in  (\ref{property}). 
Substituting (\ref{property}) in the first two equations of (\ref{eqsolitonwarped}), a straighforward computation implies that  (\ref{system_characterization}) holds.
\end{proof}

In order to analyse the system (\ref{system_characterization}), we will consider separately the cases where the potential function $f$ depends or not on the fiber. We observe that when the warping function $h$ is constant, the warped product reduces to the Riemannian or semi-Riemannian product. In this case, the base and the fiber must be Ricci solitons, as we can easily see from (\ref{system_characterization}). So, from now on, we will assume that $h$ is not constant.

For the  proof of  Theorem \ref{withdependence}, we will need the following 
lemma.

\begin{lemma}\label{important_lemma} 
Let $B^n \times F^m$ be a product manifold and $h:B^n\rightarrow \mathbb{R},\; \varphi:F^m\rightarrow \mathbb{R}$ non constant differentiable functions. Let $\mu_1,\, \rho_1:D\subset B\rightarrow \mathbb{R}$ and  $\mu_2,\, \rho_2:G\subset F\rightarrow \mathbb{R}$ be differentiable functions, such that $D\times G$ is connected. Then 
\begin{equation}\label{difference_encoded}
h(p)\mu_2(q)+\varphi(q) \mu_1(p)=\rho_1(p)+\rho_2(q), \quad \forall 
(p,q)\in D\times G. 
\end{equation} 
if, and only if, there are constants $b,\ \tilde{b},\ c,\ \tilde{c}\in \mathbb{R}$ such that 
\begin{equation}\label{partial_system}
	\left\{
	\begin{array}[pos]{lll}
	\mu_{1}=ch+\tilde{c},\\\noalign{\smallskip}
	\rho_{1}=-bh+\tilde{b},\\\noalign{\smallskip}
	\mu_{2}=-c\varphi-b,\\\noalign{\smallskip}
	\rho_{2}=\tilde{c}\varphi-\tilde{b},
	\end{array}
	\right.
	\end{equation}
for all  $p\in D$ and $q\in G$. 	
\end{lemma} 

\begin{proof}[\textbf{Proof.}]
Assume that the relation (\ref{difference_encoded}) holds.  
Since  $h$ and $\varphi$ are not constant, we consider $(p_{0},q_{0})\in D\times G$ such that $p_{0}$ and $q_{0}$ are regular points of the functions $h$ and $\varphi$, respectively. Then there exists a vector field $X_1$ on a connected neighborhood $D_1\subset D$ of $p_0$ and a vector field $U_1$ on a connected neighborhood $G_1\subset G$
of $q_0$ such that
\[
X_1(h)(p)\neq 0, \quad U_1(\varphi)(q)\neq 0,\quad \forall p\in D_1,\; q\in G_1.
\]
Consider $X_1, X_2,...,X_n$ and $U_1,U_2,...,U_m$ orthogonal frames locally defined in (neighborhoods that we still denote by) $D_1$ and $G_1$ respectively. Applying the vector fields $X_k$, $k=1,...,n$ and $U_\alpha$, $\alpha=1,...,m$ to the relation (\ref{difference_encoded}) we get that 
\begin{equation}\label{difference_encoded_derived_again}
X_k(h)U_\alpha(\mu_2)=-X_k(\mu_1)U_\alpha(\varphi), \qquad\forall k,\, \alpha .
\end{equation}
In particular, we have 
\[
\frac{ X_1(\mu_1) }{X_1(h)}=-\frac{U_1(\mu_2)}{U_1(\varphi)}=c, \quad  \mbox{ in } D_1 \mbox{ and } G_1, 
\]
for some constant $c\in\mathbb{R}$. Hence 
\begin{equation}\label{constantc}
X_1(\mu_1)=cX_1(h)\quad \mbox{ in } D_1 \mbox{\quad and \quad}
U_1(\mu_2)=-cU_1(\varphi)  \quad \mbox{ in } G_1. 
\end{equation}
We want to show that this expression holds for all $X_i$ and $U_\alpha$. 
Fix $p_1\in D_1$ and consider $X_i(h)(p_1)$ for $i\geq 2$. 
If $X_i(h)(p_1)\neq 0$, shrinking $D_{2}$ if necessary, we can assume that $X_i(h)\neq 0$ in $D_1$. Then it follows from (\ref{difference_encoded_derived_again}) and (\ref{constantc}) that in $D_{1}$
\[
\frac{X_i(\mu_1)}{X_i(h)}=-\frac{U_1(\mu_2)}{U_1(\varphi)}=c. 
\]  
Therefore,  
\[ 
X_i(\mu_1)=cX_i(h) \qquad \mbox{ in } D_{1}.
\] 
If $X_i(h)(p_1)= 0$, then it follows from (\ref{difference_encoded_derived_again}) that $U_1(\varphi)X(\mu_1)(p_1)=0$ therefore $X_i(h)(p_1)=cX_i(\mu_1)(p_1)$. We conclude that for all $i$ and $\alpha$  we have  
\[
X_i(\mu_1-ch)=0  \qquad \mbox{ in }D_{1}.
\]
Similarly,we get that 
\[
U_\alpha(\mu_2+c\varphi)=0  \qquad \mbox{ in }G_{1}.
\]

From the last two expressions we conclude that there exist constants 
$\tilde{c},b\in\mathbb{R}$ such that 
\[
\mu_1-ch=\tilde{c}, \quad \mbox{ in } D_{1}  \qquad \mu_2+c\varphi=-b,\quad \mbox{ in } G_{1}.
\] 
It follows from (\ref{difference_encoded}) that 
\[
\rho_1+bh=\tilde{c}\varphi-\rho_2=\tilde{b}.
\]
Therefore, we obtained (\ref{partial_system}) in $D_{1}\times G_{1}$.

If there is $p_{1}\in D\backslash D_{1}$, using (\ref{difference_encoded}) in $p_{1}$ and (\ref{partial_system}) in $q\in G_{1}$ we have
\begin{equation*}
\begin{array}[pos]{lll}
\varphi(q)(-ch(p_{1})+\mu_{1}(p_{1})-\tilde{c})=\rho_{1}(p_{1})+bh(p_{1})-\tilde{b},
\end{array}
\end{equation*}
for all $q\in G_{1}$. Applying $X_{1}$ on the above identity and how $\varphi$ is not constant on $G_{1}$ it follows that (\ref{partial_system}) holds on $D\times G_{1}$. Analogously if there is $q_{1}\in G\backslash G_{1}$, we can use (\ref{difference_encoded}) in $q_{1}$, (\ref{partial_system}) in $p\in D_{1}$ and the non constancy of $h$ on $D_{1}$ to prove (\ref{partial_system}) on whole $D\times G$.
\end{proof}

\begin{proof}[\textbf{Proof of Theorem \ref{withdependence}.} ]
If $(B^n\times_h F^m, g,f,\lambda)$ is a Ricci almost soliton then it follows from Theorem \ref{decompose} that $f=\beta+h\varphi$ and the system (\ref{system_characterization}) is satisfied. We are assuming that  $h$ is not constant and $f$ depends on the fibers. Hence  $\varphi$ is not constant. 

Considering the system (\ref{system_characterization}) evaluated at   pairs of orthogonal vector fields $(X,Y)$, $X,Y\in \mathfrak{X}(B)$  and $(U,V)$, $U,V\in \mathfrak{X}(F)$   locally defined  on a neighborhood of any point $(p,q)\in B\times F$, we have    
\begin{equation}\label{two_eq_orthogonal}
\left\{ 
\begin{array}[pos]{l}
Ric_B(X,Y)+\nabla_{B}\nabla_{B}\beta(X,Y)+(\varphi-mh^{-1})\nabla_{B}\nabla_{B}h(X,Y)=0,\\\noalign{\smallskip}
Ric_F(U,V)+h\nabla_{F}\nabla_{F}
\varphi(U,V)=0.
\end{array}
\right.
\end{equation}

Fix $p_1\in B$ and consider an open neighborhood $G_1\subset B$ of regular points $q$ of $\varphi$ and $W$  a vector field such that  $W(\varphi)\neq  0$ in $G_1$. Considering the first equation of \eqref{two_eq_orthogonal} at the points $(p_1,q)$ and  applying $W$ to this equation, we get that 
\[
  	\begin{array}[pos]{l}
	\nabla_{B}\nabla_{B}h(X,Y)(p_1)=0, \\ \noalign{\smallskip}
	Ric_{B}(X,Y)(p_1)+\nabla_{B}\nabla_{B}\beta(X,Y)(p_1)=0
\end{array}\qquad \forall p_1\in B.
\]

Similarly, by fixing $q_1\in F$ and considering an open neighborhood  $D_1\subset B$, of regular points  $p$ of $h$, we obtain from the second equation of (\ref{two_eq_orthogonal}) that 
\[
\begin{array}[pos]{l}
\nabla_{F}\nabla_{F}\varphi(U,V)(q_1)=0, \\ \noalign{\smallskip}	Ric_{F}(U,V)(q_1)=0 
	\end{array} \qquad \forall q_1\in F. 
\]

Therefore, for any pairs of orthogonal vector fields $(X,Y)$ and $(U,V)$, locally defined in $B\times F$,  we have   
\begin{equation}\label{complement}
	\left\{
	\begin{array}[pos]{lll}
	\nabla_{B}\nabla_{B}h(X,Y)=0,  \\ \noalign{\smallskip}
	Ric_{B}(X,Y)+\nabla_{B}\nabla_{B}\beta(X,Y)=0,\\\noalign{\smallskip}	\nabla_{F}\nabla_{F}\varphi(U,V)=0, \\ \noalign{\smallskip}
	Ric_{F}(U,V)=0.
	\end{array}
	\right.
	\end{equation}

Let  $(p_0,q_0)\in B\times F$ such that $p_0$ and $q_0$ are regular points of the functions $h$ and $\varphi$ respectively. Then there exist vector fields $X_1$ and $U_1$ defined 
one open connected sets $D\subset B$ and $G\subset F$ with $p_0\in D$ and $q_0\in G$, such that 
\begin{equation}  \label{X1U1}
X_1(h)(p)\neq 0, \quad \forall p\in D, \qquad \qquad U_1(\varphi)(q)\neq 0,\quad 
\forall q\in G.
\end{equation}
Let $\{X_1,X_j\}_{j=2}^n$  and  $\{U_1,U_\alpha\}_{\alpha=2}^m$ be orthogonal vector fields on $D$ and $G$ respectively. Without loss of generality we may consider  
\begin{equation}\label{initial_condition}
\left\{ 
\begin{array}[pos]{lll}
g_{B}(X_{j},X_{k})=\epsilon_{j}\delta_{jk}h^2, \qquad \forall j,\ k\in\{1,\ldots,n\},\\\noalign{\smallskip}
g_{F}(U_{\alpha},U_{\gamma})=\varepsilon_{\alpha}\delta_{\alpha\gamma}, \; \qquad \forall \alpha,\ \gamma\in\{1,\ldots,m\}, 
\end{array}
\right.
\end{equation}
where $\epsilon_{j}$ and $\varepsilon_{\alpha}$ denote the signatures of the vector fields. 

Now we consider the system (\ref{system_characterization}) evaluated at the  pairs $(X_j,X_j)$  and $(U_\alpha, U_\alpha)$. Subtracting the first equation multiplied by $\epsilon_j$ from the second one mutiplied by 
$\varepsilon_\alpha$, we get the following expression 
\begin{equation}\label{difference_encoded_j_alpha}
\varphi(q)\mu_{1j}(p)+h(p)\mu_{2\alpha}(q)=\rho_{1j}(p)+\rho_{2\alpha}(q), \qquad \forall (p,q)\in D\times G, 
\end{equation}
where $1\leq j\leq n,\quad 1\leq \alpha \leq m$ and  
\begin{equation}\label{key}
\left\{
\begin{array}[pos]{lll}
\mu_{1j}=-\epsilon_j \nabla_B\nabla_Bh (X_{j},X_{j})+h|\nabla_B h|^2, \\ \noalign{\smallskip}
\rho_{1j}=h\Delta_B h +(m-1)|\nabla_{B}h|^{2}-h(\nabla_B h) \beta +
\epsilon_{j}[Ric_{B}+\nabla_{B}\nabla_{B}\beta -mh^{-1}\nabla_{B}\nabla_{B} h ](X_{j},X_{j}), 
\\\noalign{\smallskip}
\mu_{2\alpha}=\varepsilon_{\alpha}\nabla_{F}\nabla_{F}\varphi(U_{\alpha},U_{\alpha}),\\ \noalign{\smallskip}
\rho_{2\alpha}=-\varepsilon_{\alpha}Ric_{F}(U_{\alpha},U_{\alpha}).
\end{array}
\right.
\end{equation}

In view of Lemma \ref{important_lemma}, it follows from (\ref{difference_encoded_j_alpha}) that, for each pair $(j, \alpha)$, there exist contants  $a_{j\alpha}, b_{j\alpha}, c_{j\alpha}, d_{j\alpha}$, 
such that 
\begin{equation}\label{partial_system_j+alpha}
	\left\{
	\begin{array}[pos]{lll}
	\mu_{1j}=c_{j\alpha}h+\tilde{c}_{j\alpha},\\\noalign{\smallskip}
	\rho_{1j}=-b_{j\alpha}h+\tilde{b}_{j\alpha},\\\noalign{\smallskip}
	\mu_{2\alpha}=-c_{j\alpha}\varphi-b_{j\alpha},\\\noalign{\smallskip}
	\rho_{2\alpha}=\tilde{c}_{j\alpha}\varphi-\tilde{b}_{j\alpha}.
	\end{array}
	\right.
	\end{equation}
Therefore,
\[
\frac{X_1(\mu_{1j})}{X_1(h)}=c_{j\alpha},\qquad 
\frac{X_1(\rho_{1j})}{X_1(h)}=-b_{j\alpha}, 
\]
\[
\frac{U_1(\mu_{2\alpha})}{U_1(\varphi)}=-c_{j\alpha},\qquad 
\frac{U_1(\rho_{2\alpha})}{U_1(\varphi)}=\tilde{c}_{j\alpha}, 
\]
i.e.,  $c_{j\alpha}$, $b_{j\alpha}$ do not depend on $\alpha$, 
and $c_{j\alpha}$ and $\tilde{c}_{j\alpha}$ do not depend on $j$.
Hence we denote $c_{j\alpha}=c$, $b_{j\alpha}=b_j$ and $\tilde{c}_{j\alpha}=\tilde{c}_\alpha$. Moreover, 
it follows from  (\ref{partial_system_j+alpha}) that 
\[
\mu_{1j}-ch=\tilde{c}_{\alpha} \qquad \mbox{ and } \qquad
 \mu_{2\alpha}+c\varphi=-b_{j}. 
 \]
Therefore, $\tilde{c}_\alpha$ does not depend on $\alpha$ and $b_j$ does not depend on $j$. Hence we may denote $\tilde{c}_\alpha=\tilde{c}$ , $b_j=b$ and 
\[
\rho_{1j}+bh=\tilde{b}_{j\alpha}, \qquad \qquad \rho_{2\alpha}-\tilde{c}\varphi=-\tilde{b}_{j\alpha}.
\]
We conclude that $\tilde{b}_{j\alpha}$ does not depend on $j$ and $\alpha$ and we can denote $\tilde{b}_{j\alpha}=\tilde{b}$. Therefore, it follows from (\ref{key}) and  (\ref{partial_system_j+alpha})  that in $D\times G$ we have 
\begin{equation}\label{key2}
\left\{
\begin{array}[pos]{l}
-\epsilon_j \nabla_B\nabla_Bh (X_{j},X_{j})+h|\nabla_B h|^2=ch+\tilde{c},\\\noalign{\smallskip}
h\Delta_B h +(m-1)|\nabla_{B}h|^{2}-h(\nabla_B h) \beta +
\epsilon_{j}[Ric_{B}+\nabla_{B}\nabla_{B}\beta -mh^{-1}\nabla_{B}\nabla_{B} h ](X_{j},X_{j})=-bh+\tilde{b},
\\\noalign{\smallskip}
\varepsilon_{\alpha}\nabla_{F}\nabla_{F}\varphi(U_{\alpha},U_{\alpha})=
-c\varphi-b\\\noalign{\smallskip}
-\varepsilon_{\alpha}Ric_{F}(U_{\alpha},U_{\alpha})=\tilde{c}\varphi-\tilde{b}.
\end{array}
\right.
\end{equation}
Considering (\ref{complement})  for  the  orthogonal vector fields $\{X_j\}_{j=1}^n$, $\{U_\alpha\}_{\alpha=1}^m$  it follows from (\ref{key2})  that in $D\times G$ we have 
\begin{equation}\label{pre_system_withdependence2}
\left\{
\begin{array}[pos]{l}
\nabla_B\nabla_Bh +\left[ch^{-1}+\tilde{c}h^{-2}-h^{-1}|\nabla_Bh|^2\right]g_B=0, \\ \noalign{\smallskip}
Ric_{B}+\nabla_{B}\nabla_{B}\beta+\left\{ h\Delta_B h -|\nabla_{B}h|^{2}-h(\nabla_B h) \beta-\tilde{b}+bh+m\tilde{c}h^{-1}+mc\right\}h^{-2}g_B=0, \\ \noalign{\smallskip}
\nabla_{F}\nabla_{F}\varphi+(c\varphi +b)g_F=0, \\ \noalign{\smallskip}
Ric_F+(\tilde{c}\varphi -\tilde{b})g_F=0, 
\end{array}
\right.
\end{equation}

We will now prove that (\ref{pre_system_withdependence2}) holds in $B\times F$. Let $p_1\in B$ and $X\in T_{p_1}B$ such that $g_B(X,X)=\epsilon_X h^2(p_1)$, where $\epsilon_X=\pm 1$.
Consider $q\in G$ and the system (\ref{system_characterization}) at the pair of vectors $(X,X)$ and the pair of vectors fields $(U_1,U_1)$ at $(p_1,q)$, $q\in G$. Multiplying the first equation by $-\epsilon_X$ and adding to the second equation multiplied by $\varepsilon_1$, we get 
\begin{equation}\label{differenceXU1}
\varphi(q)\mu_{1X}(p_1)+h(p_1)\mu_{21}(q)=\rho_{1X}(p_1)+\rho_{21}(q),
\qquad \forall q\in G, 
\end{equation}
where 
\begin{equation}\label{keyXU1}
\left\{
\begin{array}[pos]{lll}
\mu_{1X}=-\epsilon_X \nabla_B\nabla_Bh (X,X)+h|\nabla_B h|^2, \\ \noalign{\smallskip}
\rho_{1X}=h\Delta_B h +(m-1)|\nabla_{B}h|^{2}-h(\nabla_B h) \beta +
\epsilon_{j}[Ric_{B}+\nabla_{B}\nabla_{B}\beta -mh^{-1}\nabla_{B}\nabla_{B} h ](X,X), 
\\\noalign{\smallskip}
\mu_{21}=-c\varphi -b,\\ \noalign{\smallskip}
\rho_{21}=\tilde{c}\varphi -\tilde{b}, 
\end{array}
\right.
\end{equation}
where the last two equalities follow from (\ref{key2}) and  the fact that $q\in G$.  Therefore, (\ref{differenceXU1}) reduces to 
\[
[-ch(p_1)+\mu_{1X}(p_1)-\tilde{c}]\varphi(q)=bh(p_1)+\rho_{1X}(p_1)-\tilde{b}, \qquad \forall q\in G.
\]
Applying the vector field $U_1$ to this equation, we conclude that
\begin{equation}\label{mu1Xrho1X}
\mu_{1X}(p_1)=ch(p_1)+\tilde{c}, \qquad \rho_{1X}(p_1)=-bh(p_1)+\tilde{b}. 
\end{equation}

Similarly,  considering  $q_1\in F$ and $U\in T_{q_1}F$ such that 
$g_F(U,U)=\varepsilon_U= \pm 1$, for all $p\in D$ the equations of 
(\ref{system_characterization}) evaluated at the pairs $(X_1,X_1)$ and 
$(U,U)$ will imply that 
\[
\varphi(q_1)\mu_{11}(p)+h(p)\mu_{2U}(q_1)=\rho_{11}(p)+\rho_{2U}(q_1),
 \]
 where  
 \[
 \mu_{2U}=\varepsilon_U \nabla_F\nabla_F\varphi(U,U),  \qquad 
 \rho_{2U}=-\varepsilon_U Ric_F(U,U).
 \]
Analogue arguments as before will imply that 
\begin{equation}\label{mu2Urho2U}
\mu_{2U}(q_1)=-c\varphi(q_1)-b,\qquad \rho_{2U}(q_1)=\tilde{c} \varphi(q_1)-\tilde{b}.
\end{equation}

Since $p_1\in B$ and $q_1\in F$ are arbitrary, we conclude that for any locally defined vector fields $X\in \mathfrak{X}(B)$ and $U\in\mathfrak{X}(F)$, such that $g_B(X,X)=\epsilon_X h^2$ and $g_F(U,U)=\varepsilon_U $ we have that (\ref{mu1Xrho1X}) and (\ref{mu2Urho2U}) hold. We now  consider any  point $(p_1,q_1)\in B\times F$ and  orthogonal  fields  locally defined $Y_1, ...,Y_n$ in $\mathfrak{X}(B)$, $V_1,...V_m$ in $\mathfrak{X}(F)$ such that $g_B(Y_j,Y_j)=\epsilon_j h^2$ and $g_F(V_\alpha, V_\alpha)=\varepsilon_\alpha$. Then  
\[
\left\{
\begin{array}[pos]{l}
-\epsilon_j \nabla_B\nabla_Bh (Y_{j},Y_{j})+h|\nabla_B h|^2=ch+\tilde{c},\\\noalign{\smallskip}
h\Delta_B h +(m-1)|\nabla_{B}h|^{2}-h(\nabla_B h) \beta +
\epsilon_{j}[Ric_{B}+\nabla_{B}\nabla_{B}\beta -mh^{-1}\nabla_{B}\nabla_{B} h ](Y_{j},Y_{j})=-bh+\tilde{b},
\\\noalign{\smallskip}
\varepsilon_{\alpha}\nabla_{F}\nabla_{F}\varphi(V_{\alpha},V_{\alpha})=
-c\varphi-b\\\noalign{\smallskip}
-\varepsilon_{\alpha}Ric_{F}(V_{\alpha},V_{\alpha})=\tilde{c}\varphi-\tilde{b}.
\end{array}
\right.
\]
Considering (\ref{complement})  for  the  orthogonal vector fields 
$\{Y_j\}_{j=1}^n$ and $\{V_\alpha\}_{\alpha=1}^m$  it follows that (\ref{pre_system_withdependence2}) holds 
in $B\times F$.

\vspace{.1in}

We will now use Bochner formula (\ref{bochner}) to prove that  
\begin{equation}\label{dat}
\tilde{c}=0,\qquad \qquad \tilde{b}=(m-1)c.
\end{equation}
In fact, it follows from the third equation of (\ref{pre_system_withdependence2}) that 
\[
U_1(\Delta_F\varphi)=-cm U_1(\varphi).
\]
From  the fourth equation, we have 
\[
	Ric_{F}(\nabla_{F}\varphi,U_1)=(-\tilde{c}\varphi+\tilde{b})U_1(\varphi).
\]	
Moreover,  
	\[ 
	\begin{array}[pos]{lll}
	div(\nabla_{F}\nabla_{F}\varphi)(U_1)&=&\displaystyle\sum_{\alpha=1}^{m}(\nabla_{F_{\ U_{\alpha}}}\nabla_{F}\nabla_{F}\varphi)(U_1,U_{\alpha})\\\noalign{\smallskip}
	&=&\displaystyle\sum_{\alpha=1}^{m}(\nabla_{F_{\ U_{\alpha}}}(-(c\varphi+b)g_{F}))(U_1,U_{\alpha})\\\noalign{\smallskip}
	&=&\displaystyle-cg_{B}(U_1,\sum_{\alpha=1}^{m}U_{\alpha}(\varphi)U_{\alpha})\\\noalign{\smallskip}
	&=&-cU_1(\varphi).
	\end{array}
	\]
Now  Bochner formula implies that 
\[ \label{consequenceofbochner}
	[\tilde{c}\varphi-\tilde{b}+c(m-1)]U_1(\varphi)=0.
	\] 
Since $U_1(\varphi)\neq0$, we conclude that (\ref{dat}) holds.

Therefore, on $B\times F$,  the system (\ref{pre_system_withdependence2}) reduces to 
\begin{equation}\label{pre_system_withdependence3}
\left\{
\begin{array}[pos]{l}
\nabla_B\nabla_Bh +(c-|\nabla_Bh|^2 )h^{-1}g_B=0, \\ \noalign{\smallskip}
Ric_{B}+\nabla_{B}\nabla_{B}\beta+
\left\{h^{-1}\left[ \Delta_B h-(\nabla_Bh)\beta+b\right]+h^{-2}(c-|\nabla_{B}h|^{2})\right\}
g_B=0, \\ \noalign{\smallskip}
\nabla_{F}\nabla_{F}\varphi+(c\varphi +b)g_F=0,\\ \noalign{\smallskip}
Ric_F-(m-1)cg_F=0.
\end{array}
\right.
\end{equation}

Observe that for any $X\in \mathfrak{L}(B)$, we have the following expressions 
\[
\begin{array}{l}
\nabla_B\nabla_Bh(X,\nabla_BX)=g_{B}(\nabla_X\nabla_Bh,\nabla_Bh)=\frac{1}{2}X(|\nabla_Bh|^2), \\ \noalign{\smallskip}
\nabla_B\nabla_Bh(X,\nabla_BX)=\left( |\nabla_B h|^2-c\right)h^{-1}X(h), 
\end{array}
\]
where the second equality follows from (\ref{pre_system_withdependence3}). 
Therefore, 
\[
\frac{X(|\nabla_Bh|^2)}{2}-\left( |\nabla_B h|^2-c\right)h^{-1}X(h)=0, 
\]
which implies that 
\[
X\left[\left(c-|\nabla_Bh|^2\right)h^{-2}\right]=0.
\]
Hence there exists a constant $a$ such that 
\[
\left(c-|\nabla_Bh|^2\right)h^{-2}=a, 
\]
i.e., (\ref{relation}) holds.  Moreover, the first equation of (\ref{pre_system_withdependence3})
reduces to 
\[
\nabla_B\nabla_Bh+ahg_B=0 
\]
and $\Delta_B h=-anh$. Hence the second equation of (\ref{pre_system_withdependence3})
reduces to
\[
Ric_B+\nabla_B\nabla_B\beta=\left[(n-1)a+h^{-1}(\nabla_Bh)\beta-bh^{-1}\right]g_B.
\]
 Finally, it follows from these  two last equations that  the first equation   
 of (\ref{system_characterization}) provides   
\[
\lambda= h^{-1}(\nabla_Bh)\beta+(m+n-1)a-bh^{-1}-ah\varphi.
\]

Therefore, the functions $f,\,h$ and $\lambda$ satisfy the system (\ref{system_withdependence}). The converse is a straightforward 
computation. 
This concludes the proof of Theorem \ref{withdependence}. 
\end{proof}

\begin{proof}[\textbf{Proof of Corollary \ref{nondependenceontehfiberriem}}]
	Suppose by contradiction that $f$ depends on the fiber, then it follows from Theorem \ref{withdependence} that $f=\beta+h\varphi$ where $\varphi$ is not constant. Moreover, $\beta,\, h,\, \varphi$ and $\lambda$ satisfy
	(\ref{system_withdependence})-(\ref{relation}). Hence there exists a vector field $U\in\mathfrak{L}(F)$ such that $U(\varphi)\neq 0$ on an open subset of $F$. Since  $\lambda$ is constant, taking the derivative of (\ref{lambda}) with respect to $U$, we obtain
	$0=U(\lambda)=-ahU(\varphi)$. Hence $a=0$ and the first equation of (\ref{system_withdependence}) reduces to
	$\nabla_{B}\nabla_{B}h=0$. However, it follows  from Proposition   \ref{nondependenceontehfiber} that if $B\times_h F$ is complete then 
	$\nabla_Bh$ is not parallel, which is a contradiction.
	
\end{proof}

\begin{proof}[\textbf{Proof of Theorem \ref{withoutdependence}}]
It follows from Proposition \ref{decompose} that if $(B^n\times_hF^m, g,f,\lambda)$ is a Ricci almost soliton and $f$ is constant on $F$, then in the decomposition of $f$ given by (\ref{property}) we may consider $\varphi=0$. Therefore, from  the first equation  of (\ref{system_characterization}) we get that the first equation of (\ref{system_withoutdependence}) holds  and  that $\lambda$ is a function constant on $F$, hence it depends only on $B$. 
In order to obtain the other equations of 
(\ref{system_withoutdependence}), we 
 observe that if $U\in\mathfrak{L}(F)$ is a unitary vector field satisfying $g_{F}(U,U)=\epsilon\in\{-1,1\}$ we obtain from the second equation of 
 (\ref{system_characterization}) :
\begin{equation*}
\epsilon Ric_{F}(U,U)=h\Delta_{B}h+(m-1)|\nabla_{B}h|^{2}-h(\nabla_{B}h)\beta+\lambda h^{2}.
\end{equation*}
Since the left hand side is a function defined only on $F$ and the right hand side is a function defined only on $B$, there is a constant $\tilde{c}\in\mathbb{R}$ independent of the fixed field $U$, (as we can see using the right hand side of the above equality), such that
\begin{equation*}
\lambda h^{2}=h(\nabla_{B}h)\beta-(m-1)|\nabla_{B}h|^{2}-h\Delta_{B}h+\tilde{c},
\end{equation*}
and 
\begin{equation*}
Ric_{F}=\tilde{c}g_{F}.
\end{equation*}
In order to normalize the Einstein constant, we consider $\tilde{c}=(m-1)c$. This proves that (\ref{system_withoutdependence}) holds.  
The converse is a simple calculation. 
\end{proof}

\begin{proof}[\textbf{Proof of Theorem \ref{impropercase}}]
From Theorem \ref{withdependence}, we have that $f=\beta+h\varphi$ 
and equations (\ref{system_withdependence})-(\ref{relation}) are satisfied. If $\nabla_{B}h$ is an improper vector field on $B$, it follows from equation (\ref{relation})   that $a=c=0$. Hence, (\ref{system_withdependence}) and (\ref{lambda}) imply that 
  $\nabla_{B}h$ is a parallel light like vector field, $(F,g_{F})$ is Ricci flat and 
	\begin{equation}\label{h_propercase}
	\left\{ 
	\begin{array}[pos]{ll}
	Ric_{B}+\nabla_{B}\nabla_{B}\beta=\lambda g_{B}, \\\noalign{\smallskip}
	\lambda=h^{-1}[(\nabla_{B}h)\beta-b], \\\noalign{\smallskip}
	\nabla_{F}\nabla_{F}\varphi+bg_{F}=0. 
	\end{array}
	\right.
	\end{equation}
	Now we will prove that $\lambda$ is constant. If $\lambda=0$ there is nothing to prove.
	Otherwise there is an open set $U\subset M$ where $\lambda$ does not vanish. Then it follows from the second equation of (\ref{h_propercase}) that 
	\begin{equation}\label{lnlam}
	\begin{array}[pos]{lll}
	\displaystyle\frac{1}{2}X(\ln(\lambda^{2}))&=&\displaystyle\frac{1}{2}X(\ln(h^{-2}[g(\nabla_{B}h,\nabla_{B}\beta)-b]^{2})) \\\noalign{\smallskip}
	&=&-h^{-1}X(h)+[g(\nabla_{B}h,\nabla_{B}\beta)-b]^{-1}X(g(\nabla_{B}h,\nabla_{B}\beta)) \\\noalign{\smallskip}
	&=&-h^{-1}X(h)+[g(\nabla_{B}h,\nabla_{B}\beta)-b]^{-1}
	\nabla_{B}\nabla_{B}\beta(X,\nabla_B h).
		\end{array}
	\end{equation}
Since $\nabla_{B}h$ is a parallel vector field, Bochner's Formula implies that $Ric(X,\nabla_{B}h)=0$, hence from the first equation of (\ref{h_propercase}), we get that 
$\nabla_{B}\nabla_{B}\beta(X,\nabla_B h)	=\lambda g_B(X,\nabla_B h)$.
We conclude, using the second equation of (\ref{h_propercase}) that \eqref{lnlam} reduces to
\[
\frac{1}{2}X(\ln(\lambda^{2})) 	=-h^{-1}X(h)+h^{-1}X(h)=0,
\]
which proves that $\lambda$ is constant. The converse is immediate. 
	
	Now suppose that $(F,g_{F})$ is complete. Since $\nabla_{F}\nabla_{F}\varphi+bg_{F}=0$, the result follows from Theorem \ref{homothetic}  in the Appendix, if $b\neq0$ and from Theorem \ref{conformal2} if $b=0$. 
\end{proof}

\begin{proof}[\textbf{Proof of Theorem \ref{propercase}}]
If $(B^n\times_h F^m, g,f,\lambda)$ is a Ricci almost soliton with $h$ non constant and $f$ depending on the fiber then, it follows from  Theorem \ref{withdependence} that there are functions $\beta:B\rightarrow\mathbb{R}$ and  $\varphi:F\rightarrow\mathbb{R}$ and constants $a,b,c\in\mathbb{R}$, such that $f=\beta+h\varphi$ where $\beta,\, h,\, \varphi$ and $\lambda$ satisfy (\ref{system_withdependence})-(\ref{relation}).

If $\nabla_{B}h$ is a homothetic vector field, then $a=0$. It means that this vector field is parallel, and by the same argument as in the proof of Theorem \ref{impropercase}, we see that $\lambda$ is constant, which proves that $(B^n\times_h F^m, g,f,\lambda)$ is a Ricci soliton.

From now on we will suppose that $\nabla_{B}h$ is a non homothetic vector field, that is, that $a\neq0$.

If $n=1$, from (\ref{system_withdependence})-(\ref{relation}) we get  
	\begin{equation*}
	\left\{ 
	\begin{array}[pos]{ll}
	h''\pm ah=0,\\\noalign{\smallskip}
	\pm(h')^{2}+ah^{2}=c,\\\noalign{\smallskip}
	Ric_{F}=(m-1)cg_{F},
	\end{array}
	\right.
	\end{equation*}
where $g_{B^1}=\pm dt^2$. Therefore, $B^{1}\times_{h}F^{m}$ is an Einstein manifold with normalized Einstein constant $a$, as a consequence of Corollary \ref{einsteinonedimentional}.

If $n\geq2$, it follows from the second equation of (\ref{system_withdependence}) that $(B,g_{B},\beta,\bar{\lambda})$ is a Ricci almost soliton, i.e.,  
\begin{equation}\label{eqsolitontoB}
Ric_{B}+\nabla_{B}\nabla_{B}\beta=\bar{\lambda}g_{B}
\end{equation}
where 
\begin{equation}\label{expressiontolambdabar}
\bar{\lambda}=h^{-1}(\nabla_{B}h)\beta-bh^{-1}+(n-1)a.
\end{equation}
From the first equation of (\ref{system_withdependence}), we get that  
 $\nabla_{B}h$ is a gradient conformal field satisfying 
\begin{equation}\label{conformalgradh}
\nabla_{B}\nabla_{B}h+ahg_{B}=0, 
\end{equation}
i.e., $(B,g_B,\beta,\bar{\lambda})$ is a Ricci almost soliton. 
Moreover, $ \nabla_B\nabla_B h -\Delta_B h/n g=0$. 
 By hypothesis, $\nabla_{B}h$ is a non homothetic vector field hence 
  $\nabla_{B}h$ is a proper vector field, and therefore $a\neq 0$  
and $h$ admits regular points.  
  Fixing a regular point  of $h$, $p\in B$, 
   it follows from Proposition  \ref{localcharacterization}  
  that there exists  a connected open set $D\subset B$, containing $p$,  such that $D$ is   diffeomorphic to $(-\varepsilon,\varepsilon)\times N^{n-1}$
   for $\varepsilon>0$ and a regular level $N^{n-1}$ of $h$,  in such a way that $h$ does not depend on $N^{n-1}$ and $(D,g_{D})$ is isometric   to   $((-\varepsilon,\varepsilon)\times N^{n-1},\pm dt^{2}+h'(t)^{2}g_{N})$, where $g_{D}=g_{B}|_{D}$ and $g_{N}=g_{B}|_{N}$. By restricting $\beta$ and $\bar{\lambda}$ to $D$, we have that $(D,g_{D},\beta,\bar{\lambda})$ is a Ricci almost soliton, therefore
\begin{equation}\label{Dproduct}   
((-\varepsilon,\varepsilon)\times_{h'} N^{n-1},\pm dt^{2}+h'(t)^{2}g_{N},\beta,\bar{\lambda}),
\end{equation}
is also a Ricci almost soliton. We are going to use this coordinate system to conclude that $(D,g_{D})$ is an Einstein manifold with normalized Einstein constant $a$. This is equivalent to proving that  
following equations hold 
	\begin{equation}\label{conditiontobeeinstein}
	\left\{ 
	\begin{array}[pos]{ll}
	h'''\pm ah'=0,\\\noalign{\smallskip}
	\pm(h'')^{2}+a(h')^{2}=\bar{c},\\\noalign{\smallskip}
	Ric_{N}=(n-2)\bar{c}g_{N},  
	\end{array}
	\right.
	\end{equation}
as one can see from Corollary \ref{einsteinonedimentional}. 
In order to do so, we must consider two cases whether  $\beta$ depends  on $N^{n-1}$ or not. 

Suppose that $\beta$ depends on $N^{n-1}$, then we can apply Theorem \ref{withdependence}  to  (\ref{eqsolitontoB}), when  restricted to $D$,  given as in  \ref{Dproduct}. From the first and  fourth equations of (\ref{system_withdependence}) we get that the following equations hold 
	\begin{equation}\label{partbetadepend}
	\left\{ 
	\begin{array}[pos]{ll}
	h'''\pm\bar{a}h'=0,\\\noalign{\smallskip}
	Ric_{N}=(n-2)\bar{c}g_{N},
	\end{array}
	\right.
	\end{equation}
for some  constants $\bar{a},\bar{c}\in\mathbb{R}$. Moreover, from 
(\ref{relation})  the constants $\bar{a}$ and $\bar{c}$ are related to $h'$ by the equation $\pm(h'')^{2}+\bar{a}(h')^{2}=\bar{c}$.
It follows from the first equation of (\ref{system_withdependence})
and (\ref{partbetadepend}) 
that $a=\bar{a}$.  This proves (\ref{conditiontobeeinstein}) for this case.

Suppose that $\beta$ does not depend on $N^{n-1}$, 
 then since (\ref{eqsolitontoB}) holds, we can apply Theorem \ref{withoutdependence}  to $D$ given as in    (\ref{Dproduct}). 
Then  (\ref{system_withoutdependence}) reduces to
\begin{equation}\label{eqv} 
\begin{array}{l}
\beta''-(n-1)(h')^{-1}h'''=\pm \bar{\lambda},\\ \noalign{\smallskip}
(h')^2\bar{\lambda}=\pm h'h''\beta'\mp (n-2)(h'')^2\mp h'''+\bar{c}(n-2),\\
\noalign{\smallskip}
Ric_N=\bar{c}(n-2)g_N, 
\end{array}
\end{equation}
for some constant $\bar{c}\in \mathbb{R}$. 
Moreover, the first equation of 
 (\ref{system_withdependence}) restricted to $D$ gives $h''\pm ah=0$ and hence $h'''\pm ah'=0$. These two equations substituted into the first two equations of (\ref{eqv}) implies that 
\begin{equation} \label{eqIeII}
\begin{array}{l}
\beta''\pm(n-1)a=\pm\bar{\lambda}\\ \noalign{\smallskip}
(h')^2\bar{\lambda}=-a hh'\beta'\mp (n-2)a^2h^2+a(h')^2+\bar{c}(n-2).  
\end{array}
\end{equation} 
Substituting  (\ref{expressiontolambdabar}) into both equations of 
(\ref{eqIeII}), and using (\ref{relation}) we conclude that the following equations hold  
 	\begin{equation}\label{intermediateeinstein}
	\left\{ 
	\begin{array}[pos]{lll}
	(\beta'h^{-1})'=\mp bh^{-2},\\\noalign{\smallskip}
	c\beta'h^{-1}=bh^{-1}h'+(n-2)(\bar{c}\mp ac)(h')^{-1},\\\noalign{\smallskip}
	h'''\pm ah'=0,\\\noalign{\smallskip}
	\pm(h'')^{2}+a(h')^{2}=\pm ac,\\\noalign{\smallskip}
	Ric_{N}=\bar{c}(n-2)g_{N}
	\end{array}
	\right.
	\end{equation}
Therefore, in order to prove that (\ref{conditiontobeeinstein}) holds, we need to show the equality $\bar{c}=\pm ac$. If $c=0$ it follows from the second equation of (\ref{intermediateeinstein}) that $\bar{c}=0$. If  $c\neq0$, then we substitute the second equation of (\ref{intermediateeinstein}) into the first one  to obtain
	\begin{equation*}
	a(\bar{c}\mp ac)=0.
	\end{equation*}
which implies $\bar{c}=\pm ac$, since $a\neq0$. 
Therefore, we have proved that (\ref{conditiontobeeinstein}) also holds  when
$\beta$ does not depend on $N^{n-1}$.

Now from Proposition \ref{isoletedpoints},  we know that the set of regular points of $h$ is a dense subset of $B$, and the argument above implies that $(B,g_{B})$  is an Einstein manifold with normalized Einstein constant $a$. As a consequence we have
	\begin{equation}
	\left\{ 
	\begin{array}[pos]{lll}
	Ric_{B}=a(n-1)g_{B},\\\noalign{\smallskip}
	\nabla_{B}\nabla_{B}h+ahg_{B}=0,\\\noalign{\smallskip}
	Ric_{F}=c(m-1)g_{F},
	\end{array}
	\right.
	\end{equation}
which implies from Proposition \ref{einsteincondition} that $B\times_{h}F$ is itself an Einstein with normalized Einstein constant $a$.

From the fundamental equation (\ref{eqriccisoliton}), we obtain that
	\begin{equation*}
	\nabla\nabla f+((m+n-1)a-\lambda)g=0,
	\end{equation*}
and $\nabla f$ is a gradient conformal field on an Einstein manifold. Proposition \ref{otherprove} says that there is a constant $a_{0}$ such that
	\begin{equation}\label{lamf}
	\lambda=-af+a(m+n-1)+a_{0},
	\end{equation}
in view of $n+m\geq 2$. Hence $\nabla\nabla f+af-a_0=0$. Moreover, since 
$f$ is non constant on $F$, \eqref{lamf} implies  that $\lambda$ is not constant.  This concludes the proof of Theorem \ref{propercase}. 

\end{proof}

In order, to prove Theorem \ref{coroclassificationnew} that provides the classification of complete Ricci almost solitons whose potential function depends on the fiber, we will use the classification of Einstein manifolds carrying conformal vector fields, available in the Appendix, Theorem \ref{zeros12}, Theorem \ref{conformal3new} and Theorem \ref{homothetic}

\begin{proof}[\textbf{Proof of Theorem \ref{coroclassificationnew}}]
Suppose that $(B^{n}\times_{h}F^{m},g,f,\lambda)$ is a complete Ricci almost soliton, with $h$ non constant and $f$ depending on $F$. Then it follows from Theorem \ref{withdependence} that  there are functions $\beta:B\rightarrow\mathbb{R}$ and  $\varphi:F\rightarrow\mathbb{R}$ and constants $a,b,c\in\mathbb{R}$ such that $f=\beta+h\varphi$, where \ $\beta,\, h,\, \varphi$ and $\lambda$ satisfy (\ref{system_withdependence})-(\ref{relation}). From  Proposition \ref{nondependenceontehfiber} the completeness of $B^{n}\times_{h}F^{m}$ implies that $\nabla_{B}h$  is not a parallel vector field on $B$ and hence 
it follows from the first equation of (\ref{system_withdependence}) that 
 $a\neq 0$, therefore $\nabla_B h$ is not homothetic. 
  Applying Theorem \ref{propercase} we have that $B^{n}\times_{h}F^{m}$, $B$ and $F$  are Einstein manifolds satisfying (\ref{RicWarpBF}) for constants $a\neq 0$ and $c,a_0\in\mathbb{R}$, $\nabla_{B}h$, $\nabla_{F}\varphi$ and $\nabla f$ are conformal  vector fields satisfying (\ref{conformalfh}) and $\lambda$ is given by (\ref{lambdaT26}).

 If $n=1$ then $g_B=\pm dt^2$ and from the first equation of (\ref{system_withdependence}) and (\ref{relation}) we have that $h''\pm a h=0$  and $\pm (h')^2+ a h^2=c$ . Since $B$ is not compact it follows that  $B^1=\mathbb{R}$ and the non vanishing of $h$  
implies that  $\pm a<0$.  Therefore $h$ satisfies 
\[
\begin{array}{l}
h''-|a|h=0,\\ \noalign{\smallskip}
(h')^2-|a|h^2=\pm c, 
\end{array}
\]
and hence (\ref{hc}) holds i.e.
\[
h=\left\{ 
\begin{array}{lll}
Ae^{\sqrt{|a|}t} & \mbox{ if } & c=0,\\ \noalign{\smallskip}
\sqrt{|\frac{c}{a}|}[ \cosh(\sqrt{|a|}t+ \theta)] & \mbox{ if } & c\neq 0,
\end{array} 
 \right.
 \]
where $A\neq 0$ and $\theta\in \mathbb{R}$.

If $n\geq2$ and $m\geq2$, it follows that $B^{n}$ and $F^{m}$ are complete  Einstein manifolds satisfying (\ref{RicWarpBF}).  

Since $f$ satisfies the first equation of (\ref{conformalfh}) it follows that $\tilde{f}=f-a_0/a$ is a solution of $\nabla\nabla \tilde{f}+a \tilde{f}g=0$, therefore from Theorem \ref{conformal3new} we conclude that 
when $f$ has some critical point then 
$B\times_hF$ is  isometric  to a manifold of Class II 1 (resp Class II 2) when $a>0$  (resp. $a<0$)  
and $f$ is a height function on $S^n_\varepsilon(1/\sqrt{a})$ (resp. $H^n_\varepsilon(1/\sqrt{|a|})$ ( see Examples \ref{Pseudospheres} and 
\ref{Pseudohyperbolicspace}); when $f$ has no critical points then 
$B\times_hF$ is  isometric to a manifold of Class II 3 or 4.    

Since $h$ satisfies the second equation of (\ref{conformalfh}) then 
it follows from Theorem \ref{conformal3new} that if $h$ has no critical points then $B$ is isometric to one of the manifolds of Class II 3 or 4 and 
if $h$ has some critical point then $B$ is isometric to a manifold of Class 
II 1 or 2 according to the sign of $a$ moreover, $h$ is a height function. However, since $h$ does not vanish it induces a restriction on  the index of $B$, in fact, it follows from Proposition \ref{zeros12} that when 
$a>0$ (resp. $a<0$) $ B$ is isometric to $\mathbb{S}^n_n(1/\sqrt{a})$ 
(resp. $\mathbb{H}^n_1(1/\sqrt{|a|})$.     
  
Since $\varphi$ satisfies the third equation of (\ref{system_withdependence}), i.e. $\nabla_F\nabla_F \varphi +(c\varphi+b)g_F=0$, it follows from Theorem \ref{homothetic} that  if $c=0$ and $b\neq 0$, then 
$F$ is isometric to a semi Eulidean space $\mathbb{R}^m_\varepsilon$. 
If $c=b=0$ then Theorem \ref{conformal2} implies  that $F$ is isometric to a manifold of Class I. Finally, if $c\neq 0$ then Theorem \ref{conformal3new} implies that $F$ is isometric to a manifold of Class II 
1 (resp. Class II 2)  when $c>0$  (resp. $c<0$)  and $\varphi$ has  some critical point  while $F$ is isometric to a manifold of Class II 3 or 4 when 
$\varphi$ has no critical points.

We conclude by observing that, since $B\times F$ is complete,  in order to avoid the phenomena of Been-Buseman example  one must have  $F^m$, $m\geq 1$  positive definite (resp. negative definite) if $B$ is positive definite (resp. negative definite).  
\end{proof}

\section{Appendix: Conformal Fields}\label{conformalfields}

Let $(M^n,g)$ be a semi-Riemannian manifold of dimension  $n\geq 2$.  For a pair of constants $b,c\in\mathbb{R}$, we consider the set $ SC(M,c,b)$ of functions  $\varphi:M\rightarrow\mathbb{R}$ that satisfy  \begin{equation}\label{typeobataequation}  
\nabla\nabla\varphi+(c\varphi+b)g=0.
\end{equation} 

The vector field $\nabla\varphi$ of a smooth function $\varphi\in SC(M,c,b)$ is said to be {\it conformal}. If $c\neq0$ we can assume that $b=0$ replacing $\varphi$ by $\varphi-b/c$. In this case we denote the vector space $SC(M,c,b)$ by $SC(M,c)$.
Equation (\ref{typeobataequation}) has been largely studied since 1920. It started with Brinkman's work \cite{brinkmann} on conformal transformations between semi-Riemannian Einstein manifolds.

Before stating some classification results for complete manifolds that admit non constant solutions to  equation (\ref{typeobataequation}), we will present examples of spaces carrying such solutions. In this section, we are following  the notation used in \cite{oneill}.
\begin{example}\label{Semieuclideanspace}{\rm
		Let $\mathbb{R}^{n}_{\varepsilon}$ be the linear space $\mathbb{R}^{n}$ with the semi-Riemannian metric of index $\xi$ 
		\begin{equation*}
		\langle v,w\rangle_{\varepsilon}=\displaystyle\sum_{j=1}^{n}\varepsilon_{j}v_{j}w_{j}.
		\end{equation*}
		If $\varphi$ is a non constant solution of (\ref{typeobataequation}), then   a straightforward calculation shows that $c$ must be zero and that, for all $b\in\mathbb{R}$, a generic solution to (\ref{typeobataequation}) in $\mathbb{R}^{n}_{\varepsilon}$ is given by
		\begin{equation}\label{expression}
		\begin{array}[pos]{lll}
		\varphi(x_{1},\ldots,x_{n}) & = & -(b/2)\sum_{j=1}^{n}\varepsilon_{j}x_{j}^{2}+\langle A_{\varepsilon},x\rangle_{\varepsilon}+A_{n+1}
		\end{array}
		\end{equation}
		where $A_{\varepsilon}=(\varepsilon_{1}A_{1},\dots,\varepsilon_{n}A_{n})\in\mathbb{R}^{n}_{\varepsilon}$
		and $A_{n+1}\in \mathbb{R}$. So $\dim{(SC(\mathbb{R}^{n}_{\varepsilon},0,b))}=n+1$.
	}
\end{example}

\begin{example}\label{Pseudospheres}{\rm
		The pseudosphere \cite{oneill}, with dimension $n$ and index $\varepsilon$, is defined as
		\begin{equation*}
		\mathbb{S}^{n}_{\varepsilon}(1/\sqrt{c})=\{x\in\mathbb{R}^{n+1}_{\varepsilon};
		\langle x,x\rangle_{\varepsilon}=1/c\},\qquad \mbox{ where } c>0.
		\end{equation*} 
		It is connected if, and only if, $0\leq\varepsilon\leq n-1$ and simply connected if, and only if, $0\leq\varepsilon\leq n-2$. Furthermore, each connected component of \ $\mathbb{S}^{n}_{\varepsilon}(1/\sqrt{c})$ is a complete semi-Riemannian manifold of dimension $n$, index $\varepsilon$ and constant curvature $c$. It is not difficult to see that the functions in $SC(\mathbb{R}^{n+1}_{\varepsilon},0,0)$ with $A_{n+2}=0$ in the expression (\ref{expression}) i.e.,  $\varphi_{A_{\varepsilon}}(x)=\langle A_{\varepsilon},x\rangle_{\varepsilon}$,    provide all the functions in  $SC(\mathbb{S}^{n}_{\varepsilon}(1/\sqrt{c}),c)$. Hence $\dim{(SC(\mathbb{S}^{n}_{\varepsilon}(1/\sqrt{c}),c))}=n+1$. Note that $\varphi_{A_{\varepsilon}}(x)=\langle A_{\varepsilon},x\rangle_{\varepsilon}$ is the height function with respect to $A_{\varepsilon}$ on the pseudosphere. .
}\end{example}

\begin{example}\label{Pseudohyperbolicspace}{\rm
		Similarly to the example above, the pseudohyperbolic space \cite{oneill}, with dimension $n$ and index $\varepsilon$,  is defined as
		\begin{equation*}
		\mathbb{H}^{n}_{\varepsilon}(1/\sqrt{-c})=\{x\in\mathbb{R}^{n+1}_{\varepsilon+1};
		\langle x,x\rangle_{\varepsilon+1}=1/c\},\qquad \mbox{ where } c< 0.
		\end{equation*}
		It is connected if, and only if, $2\leq\varepsilon\leq n$ and simply connected if and only if $1\leq\varepsilon\leq n-2$. Furthermore each connected component of \ $\mathbb{H}^{n}_{\varepsilon}(1/\sqrt{-c})$ is a complete semi-Riemannian manifold of dimension $n$, index $\varepsilon$ and constant curvature  $c$. As in the previous example, the functions in $SC(\mathbb{R}^{n+1}_{\varepsilon+1},0,0)$ with $A_{n+2}=0$ in the expression (\ref{expression}) i.e., $\varphi_{A_{\varepsilon+1}}(x)=\langle A_{\varepsilon+1},x\rangle_{\varepsilon+1}$  provide all the functions in  $SC(\mathbb{H}^{n}_{\varepsilon}(1/\sqrt{-c}),c)$ and hence $\dim{(SC(\mathbb{H}^{n}_{\varepsilon}(1/\sqrt{-c}),c))}=n+1$.  Note that $\varphi_{A_{\varepsilon+1}}(x)=\langle A_{\varepsilon+1},x\rangle_{\varepsilon+1}$ is the  height function with respect to $A_{\varepsilon+1}$ on the pseudohyperbolic space.
	}
\end{example}

\begin{example}\label{Warpedproducts}
	{\rm
		Let $\pm I\times_{h}N^{n-1}$ be a warped product manifold, where $I\subset\mathbb{R}$ is a connected interval and $N^{n-1}$ is an arbitrary Riemannian or semi-Riemannian manifold. Then  a simple calculation shows that the function
		\begin{equation*}
		\varphi(s,p)=\int_{s_{0}}^{s}h(t)dt
		\end{equation*}
		solves equation (\ref{typeobataequation}), when $h$ satisfies 
		\begin{equation}\label{simpleeq}
		h''\pm ch=0.
		\end{equation}
		Hence,  $\dim{(SC(\mathbb{H}^{n}_{\varepsilon}),c,b)}\geq1$, if (\ref{simpleeq}) holds.
	}
\end{example}

\vspace{.1in}

The following Theorems \ref{homothetic}-\ref{conformal3new} are of fundamental importance in the proofs of Section \ref{Proofofthemainresults}. They provide the classification results of complete semi-Riemannian Einstein manifolds, for which the set of functions $SC(M,c,b)$ is not empty, for some constants $c$ and $b$. These theorems assert the uniqueness of the examples given above, when $\nabla\varphi$ is proper. The improper case was analized  by Brinkman  \cite{brinkmann} showing, among other things, that  $\nabla\varphi$ must be parallel. Since then spaces carrying parallel improper vector fields are called Brinkman spaces.

\begin{theorem}[\cite{kerbrat}]\label{homothetic}
	A complete semi-Riemannian manifold,  $(M^{n},g)$, with $n\geq2$, admits a non constant solution of the  equation 
	$\nabla \nabla \varphi+bg=0$ 	 
	$b\neq0$ if, and only if, it is isometric to the semi-Euclidean space $\mathbb{R}^{n}_{\varepsilon}$.
\end{theorem}

This result is a particular case of a theorem proved by Kerbrat \cite{kerbrat}, where the author classifies  spaces carrying vector fields satisfying more general equations.

\begin{theorem}\label{conformal2}
	A complete semi-Riemannian Einstein manifold,  $(M^{n},g)$, with $n\geq2$, admits a non constant solution $\varphi$ of the equation $\nabla\nabla\varphi=0$
	if, and only if, it is isometric to
	\begin{enumerate}
		\item \cite{kerbrat} $\mathbb{R}\times N^{n-1}$, where $(N,g_{N})$ is a complete semi-Riemannian Einstein manifold, if $\nabla\varphi$ is a proper vector field 
		\item \cite{brinkmann} a Brinkman space, if $\nabla\varphi$ is an improper vector field and $n\geq3$. 
	\end{enumerate}
\end{theorem}

The theorem bellow is a compilation of the classification of Einstein manifolds carrying non-homothetic conformal fields. The Riemannian case was settled essentially by Obata \cite{obata} and \cite{kanai}, while the case with positive signature was handled by Kerbrat \cite{kerbrat}.

\begin{theorem}\label{conformal3new}
	A complete semi-Riemannian Einstein manifold,  $(M^{n},g)$,  with $n\geq2$  and index $\varepsilon$,  admits a non constant solution 
	of the equation $\nabla\nabla \varphi +c\varphi g=0$	
	with $c\neq 0$ if, and only if, it is isometric to
	\begin{enumerate}
		\item $\mathbb{S}^{n}_{\varepsilon}(1/\sqrt{c})$,  when $0\leq\varepsilon\leq n-2$; the covering of $\mathbb{S}^{n}_{n-1}(1/\sqrt{c})$ when $\varepsilon=n-1$ and  the upper part of $\mathbb{S}^{n}_{n}(1/\sqrt{c})$ when $\varepsilon=n$ if $c>0$ and 
		$\varphi$ has some critical point 
		
		\item  $\mathbb{H}^{n}_{\varepsilon}(1/\sqrt{|c|})$, when  $2\leq\varepsilon\leq n-1$;  the  covering of $\mathbb{H}^{n}_{1}(1/\sqrt{|c|})$ when $\varepsilon=1$ and the upper part of $\mathbb{H}^{n}_{0}(1/\sqrt{|c|})$ when $\varepsilon=0\,$ ,  \, if $c<0$ and  $\varphi$ has some critical point; 
		\item $(\mathbb{R}\times N^{n-1},\pm dt^{2}+\cosh^{2}(\sqrt{|c|}\,t)g_{N})$, where  $(N^{n-1},g_{N})$ is a semi-Riemannian Einstein manifold, if $\varphi$ has no critical points 
		\item $(\mathbb{R}\times N^{n-1},\pm dt^{2}\pm e^{2\sqrt{|c|}\,t}g_{N})$, where  $(N^{n-1},g_{N})$ is a Riemannian Einstein manifold, if $\varphi$ has no critical points
	\end{enumerate}
\end{theorem}

For our purposes it is important to know if a height function has zeros or not. This is because height functions can occur as warping functions
and warping functions do not admit zeros. The next  proposition  reveals which hyperquadrics  admit such functions.

\begin{proposition}\label{zeros12}
	Let $\varphi_{A}:\mathbb{R}^{n+1}_{\varepsilon}\rightarrow\mathbb{R}$ be the height function with respect to   $A\in\mathbb{R}^{n+1}_{\varepsilon}$, $A\neq0$ and $n\geq2$. Then $\varphi_{A}$ has no zeros on  $\mathbb{S}^{n}_{\varepsilon}(1/\sqrt{c})$ (resp.    $\mathbb{H}^n_{\varepsilon}(1/\sqrt{-c})$  if, and only if, $\varepsilon=n$ (resp. $\varepsilon=1$) and $A$ is a space like (resp. time like) or light like vector.
\end{proposition}

\begin{proof} We first prove the proposition in the case of the sphere.  
	Since  we are considering $\mathbb{S}^{n}_{\varepsilon}(1/\sqrt{c}) \neq\emptyset$,
	we can assume $0\leq\varepsilon\leq n$, i.e.,  $\varepsilon\neq n+1$. Moreover, $\varphi_{A}$ is a linear function, hence  $\mathbb{R}^{n+1}=\text{Ker}(\varphi_{A})\oplus \text{Im}(\varphi_{A})$.
	where $\text{Ker}  (\varphi_{A})=(A)^{\perp}
	\subset\mathbb{R}^{n+1}$. Since $A\neq0$, it follows that $\dim\{\text{Ker}(\varphi_{A})\}=n\geq2$ and 
	$\dim\{\text{Im}(\varphi_{A})\}=1$.
	In what follows, we will analyze each case  according to  $A$ being a time like, space like or light like vector. We will  consider  an  appropriate 
	orthonormal basis in each case, $\{e_{1},\ldots,e_{\varepsilon},e_{\varepsilon+1},\ldots,e_{n+1}\}$ for $\mathbb{R}^{n+1}_\varepsilon$ such that $e_1,...,e_{\varepsilon}$ are time like and $e_{\varepsilon+1},\ldots,e_{n+1}$ are space like. 
	
	Suppose that $A$ is time like. In this case,  $1\leq \varepsilon\leq n$ and we choose the basis such that  $e_{\varepsilon}=A/\sqrt{|<A,A>_\varepsilon|}$. Therefore,  $e_{n+1}$ and $e_{\varepsilon}$ are orthogonal hence,  $(1/\sqrt{c})e_{n+1}\in A^{\perp}
	\cap\mathbb{S}^{n}_{\varepsilon}(1/\sqrt{c})$, i.e., 
	$\varphi_A$ has zeros on the sphere. 
	
	Suppose that $A$ is space like. We consider the basis on $\mathbb{R}^{n+1}_\varepsilon$, such that  $e_{\varepsilon+1}=A_{\varepsilon}/|A_{\varepsilon}|$.	
	If $0\leq\varepsilon\leq n-1$, then $e_{\varepsilon+1}$ and $e_{n+1}$ are orthogonal and hence 	$(1/\sqrt{c})e_{n+1}\in A^{\perp}\cap\mathbb{S}^{n}_{\varepsilon}(1/\sqrt{c})$.
	If $\varepsilon=n$ then  $A^{\perp}$ is negative definite since it is generated by $\{e_{1},\ldots,e_{n}\}$. Therefore
	$A^{\perp}\cap\mathbb{S}^{n}_{n}(1/\sqrt{c})=\emptyset$, 
	i.e., $\varphi_A$ has no zeros on the sphere. 
	
	Suppose that $A$ is light like, then  $1\leq\epsilon\leq n$ and it is not so difficult to see that there exist orthogonal vectors $V_{1}, V_{2}\in\mathbb{R}^{n+1}$ such that $V_{1}\neq 0$ is time like, $V_{2}\neq0$ is space like and $A=V_{1}+V_{2}$. We consider the basis so that  
	$e_{\varepsilon}=V_{1}/|\sqrt{<V_{1},V_1>_\varepsilon}|$ and $e_{\varepsilon+1}=V_{2}/|V_{2}|$. If $\varepsilon\leq n-1$, then $(1/\sqrt{c})e_{\varepsilon+2}\in(A^{\perp}\cap\mathbb{S}^{n}_{\varepsilon}(1/\sqrt{c})$ . Therefore, $\varphi_A$ has no zeros on the sphere if, and only if, $\varepsilon=n$. 
	
	This completes the proof for  the case  of the sphere. 
	Considering adequate changes, the  proof for the hyperbolic space is similar. 
\end{proof}

The next result is due to Kerbrat \cite{kerbrat} and  it can be found in Kuenel's paper \cite{kunel3}.
\begin{proposition}\label{isoletedpoints}\cite{kerbrat}
	Let $\varphi:M\rightarrow\mathbb{R}$ be a solution of the equation (\ref{typeobataequation}). Then the  critical points of $\varphi$ are isolated.
\end{proposition}

The local classification below can be found in \cite{brinkmann} or \cite{kunel1} and it is of fundamental inportance for the classification of complete manifolds admitting solutions to equation (\ref{typeobataequation}). 

\begin{proposition}\label{localcharacterization}\cite{brinkmann}
	Let $(M,g)$ be a Riemannian or semi-Riemannian-manifold. The following are equivalent:
	\begin{enumerate}
		\item There is a non constant solution $\varphi$ of 
		\begin{equation*}
		\nabla\nabla\varphi-(\Delta\varphi/n)g=0,
		\end{equation*} 
		in a neighborhood of a point $p\in M$ such that  $g(\nabla\varphi,\nabla\varphi)\neq 0$.
		\item There is a neighborhood $U$ of $p\in M$, a smooth function $\varphi:(-\varepsilon,\varepsilon)\rightarrow\mathbb{R}$ with $\varphi'(t)\neq0$, for all $t\in(-\varepsilon,\varepsilon)$ and a pseudo-Riemannian manifold $(N,g_{N})$ such that $(U,g)$ is isometric to the warped product
		\begin{equation*}
		((-\varepsilon,\varepsilon)\times_{\varphi'}N,\pm dt^{2}+(\varphi')^{2}g_{N}),
		\end{equation*}
		where  $sgn(g(\varphi',\varphi'))=\pm 1$.
	\end{enumerate}
\end{proposition}


\end{document}